\documentclass[11pt, table]{amsart}
\usepackage[table]{xcolor}
\usepackage{amsmath, amsthm, amscd, amssymb, amsfonts, amsxtra, amssymb, latexsym}
\usepackage{enumerate}
\usepackage{verbatim}
\usepackage{textcomp}
\usepackage{graphicx, epsfig, tikz}
\usepackage[T1]{fontenc} 
\usepackage{t1enc}
\usepackage[utf8x]{inputenc}

\usepackage{float}
\usepackage{pb-diagram}

\usepackage{tikz-cd}

\usetikzlibrary{arrows,%
                petri,%
                topaths}%
\usetikzlibrary{automata}                
\usetikzlibrary{positioning}

\renewcommand{\baselinestretch}{1} 
\newcommand{\Z}{\mathbb{Z}}
\newcommand{\C}{\mathcal{C}}
\newcommand{\D}{\mathbb{D}}
\newcommand{\Q}{\mathbb{Q}}
\newcommand{\N}{\mathbb{N}}
\newcommand{\ff}{\mathbb{F}}
\newcommand{\F}{\mathbb{F}}

\newcommand{\sk}{\smallskip}

\newcommand{\Sym}{\mathbb{S}}
\newcommand{\inner}[1]{\mbox{$\langle{#1}\rangle$}}

\newcommand{\Ham}{{\scriptscriptstyle Ham}}

\newtheorem{thm}{Theorem}[section]
\newtheorem{prop}[thm]{Proposition}
\newtheorem{lema}[thm]{Lemma}
\newtheorem{coro}[thm]{Corollary}
\theoremstyle{definition}
\newtheorem{rem}[thm]{Remark}
\newtheorem{exam}[thm]{Example}

\theoremstyle{remark}

\theoremstyle{definition}
\newtheorem{defi}[thm]{Definition}
\newtheorem{ej}[thm]{Example}

\usepackage{color}

\begin{document}
\numberwithin{equation}{section}
\title{Isometries between finite groups}
\author{Ricardo A.\@ Podest\'a, Maximiliano G.\@ Vides}
\thanks{Partially supported by CONICET, FONCyT and SECyT-UNC}

\address{Ricardo A.\@ Podest\'a, FaMAF -- CIEM (CONICET), Universidad Nacional de C\'ordoba, \newline
 Av.\@ Medina Allende 2144, Ciudad Universitaria, (5000) C\'ordoba, Rep\'ublica Argentina. \newline
{\it E-mail: podesta@famaf.unc.edu.ar}}

\address{Maximiliano G.\@ Vides \footnote{Present address:
 Departamento de Matemática, Facultad de Ingenieria Qu\'imica, Universidad Nacional del Litoral. 
Santiago del Estero 2829, (3000) Santa Fe, Rep\'ublica Argentina.}, 
FaMAF -- CIEM (CONICET), Universidad Nacional de C\'ordoba, \newline
 Av.\@ Medina Allende 2144, Ciudad Universitaria, (5000)   C\' ordoba, Rep\'ublica Argentina. \newline
{\it E-mail: mvides@famaf.unc.edu.ar}}

\begin{abstract}
We prove that if $H$ is a subgroup of index $n$ of any cyclic group $G$ 
then $G$ can be isometrically embedded in $(H^n, d_{\Ham}^n)$, thus generalizing previous results of Carlet (1998) for $G=\Z_{2^k}$ and Yildiz and \"Odemi\c{s} \"Ozger (2012) for $G=\Z_{p^k}$ with $p$ prime. 
Next, for any positive integer $q$ we define the $q$-adic metric $d_q$ in $\Z_{q^n}$ and prove that $(\Z_{q^n}, d_q)$ is isometric to $(\Z_q^n, d_{RT})$ for every $n$, where  
$d_{RT}$ is the Rosenbloom--Tsfasman metric.
More generally, we then demonstrate that any pair of finite groups of the same cardinality are isometric to each other for some metrics that can be explicitly constructed. 
Finally, we consider a chain $\C$ of subgroups of a given group and define the chain metric $d_{\C}$ and chain isometries between two chains. Let $G, K$ be groups with 
$|G|=q^n$, $|K|=q$ and let $H<G$. Using chains, we prove that under certain conditions, $(G,d_\C) \simeq (K^n, d_{RT})$ and  $(G,d_\C) \simeq (H^{[G:H]}, d_{BRT})$ where $d_{BRT}$ is the block Rosenbloom--Tsfasman metric which generalizes $d_{RT}$. 
\end{abstract}

\maketitle

\section{Introduction}
\subsubsection*{Historical background}
The Hamming metric $d_{\Ham}$ is the most classic and commonly used metric in coding theory, typically in codes defined over finite fields. Since the 90s, the Lee metric $d_{Lee}$ was also considered on the rings $\Z_m$. 
The Gray map is an isometry between $(\Z_4, d_{Lee})$ and $(\Z_2 \times \Z_2, d_{\Ham})$. 
This map naturally extends to an isometry from $\Z_4^n$ to  $\Z_2^{2n}$.
In a famous paper from 1994, Hammons et al (\cite{HKCSS}) used the Gray isometry to explain the formal duality exhibited by some pairs of binary non-linear codes such as Kerdock and Preparata codes and Goethals and Goethals--Delsarte codes (previously, Nechaev  obtained some similar results in \cite{Ne}). 
 
Few years later, Salagean-Mandache (\cite{Salagean-Mandache99}) proved that, except for the known case $p=n=2$,
it is not possible to construct a metric $d$ induced by a weight in $\Z_{p^{n}}$ such that $(\Z_{p^{n}},d)$ is isometric to 
$(\Z_p^n,d_{\Ham})$ for any prime $p$.
This result was then extended by Sueli Costa and collaborators showing the non-existence of isometries from $\Z_{m^n}$ to a 
Hamming space $X^n$, $|X|=m$ (see \cite{Symmetry-Lee} for $m=p$ prime, \cite{DBLP:journals/dm/AlvesC03} for arbitrary $m$).

In another direction, Carlet (\cite{Ca}) generalized the Gray map to an embedding between $\Z_{2^{k}}$ and $\Z_2^{2^{k-1}}$ preserving distances. This map naturally extends coordinatewise to $(\Z_{2^{k}})^n$ and $\Z_2^{2^{k-1}n}$. 
A couple of years later, Yildiz--\"Ozger (\cite{YO}) proved that $\Z_{p^k}$, with $p$ an odd prime, can be isometrically embedded into 
$\Z_p^{p^{k-1}}$ with the Hamming metric for any $k>1$ (see Remark \ref{YO recover}).
From a more general point of view, 
Greferath and Schmidt (\cite{Greferath}) further generalized the Gray map to an embedding from an arbitrary finite chain ring $R$ with the homogeneous metric to the residue field $F=R/\mathfrak{m}$ with de Hamming metric. More precisely, $(R, d_{Hom}) \hookrightarrow 
(F^{q^{m-1}}, d_{\Ham})$ where $q=|F|$ and $m={\rm length}(R)$. 
They used their map to construct interesting non-linear binary codes.
More recently, D'Oliveira and Firer (\cite{embeddinghamming}, \cite{minimalhamming}) showed that up to a decoding equivalence, any metric space can be isometrically embedded into a hypercube with the Hamming metric.

The goal of this work is to better understand isometries and isometric embeddings between finite groups (typically finite fields or finite rings for their applications in coding theory). We will give new explicit isometries and isometric embeddings from cyclic groups and also provide a general procedure to obtain isometries between arbitrary groups.

\subsubsection*{Outline and results}
Now, we briefly summarize the structure of, and the results in, the paper. In Section 2, we first recall some basic preliminaries on metric spaces, $G$-invariant metrics and isometries. If $G$ is a group acting on a metric space $(X,d)$, we define the associated symmetry group and their $G$-representations. 
In Proposition \ref{transfer} we show that given  
$(X,d)$ with a $G$-representation, there is a bijection $\varphi : X \rightarrow G$ inducing a group structure on $X$ and a metric $d_G$ on $G$ such that $\varphi$ is a group isomorphism and 
$\varphi :(X,d) \rightarrow (G,d_G)$ is an isometry.

In the next section we give a simple group-theoretical proof of the known result that there are no cyclic representations of a Hamming 
space $X^n$ for $X$ of prime cardinality (see Proposition~\ref{prop isometrias}). In particular, there is no isometry between $\Z_{p^n}$ and $(\Z_p^n,d_{\Ham})$.

In Section 4 we consider isometric embeddings, i.e.\@ injective maps between metric spaces preserving distances. We generalize the result of Yildiz--Özger asserting that $\Z_{p^k}$, $p$ prime, can be isometrically embedded into $\Z_{p}^{p^{k-1}}$ for any $k>1$ with the Hamming metric.
In Theorem~\ref{teo emb} we generalize this result by proving that for any $m$ and any subgroup $H$ of $\Z_m$ of index $n$, $\Z_m$ can be isometrically embedded into $H^n$ with the Hamming metric. This allows to isometrically embed a ring into rings of different characteristics as noted in 
Remark~\ref{chars}. 
In Example 4.8 we consider the subgroups of $\Z_{12}$. In Remark \ref{lee rec} we show that the isometric embedding of $\Z_{2n}$ into $\Z_2^n$ with the Hamming metric recovers the Lee metric on $\Z_{2n}$.

In Section 5, for any $q, n\in \N$ with $q\ge 2$, we define the $q$-adic metric $d_q$ on $\Z_{q^n}$. 
The $RT$-metric was introduced by Rosenbloom and Tsfasman in \cite{RT-metric} and has since then proven to be a quite useful metric in coding theory.
In Theorem \ref{IsoRT} we give a short and direct proof that $\Z_{q^n}$ with the $q$-adic metric is isometric to $(\Z_q)^n$ with the $RT$-metric, that is 
$$(\Z_{q^n}, d_q) \simeq (\Z_q^n, d_{RT}).$$

In the next section we show that any isometry between subgroups can be extended to the ambient groups 
(see Theorem \ref{Isonotrivial}). This implies that any pair of groups of the same size (and hence all) are isometric (see Corollary \ref{ext metric}). So, for instance, $\Z_2^3$, $\Z_2\times \Z_4$, $\Z_8$, $\D_4$ and $\Q_8$ are all mutually isometric.

In Section 7, we consider metrics on chain of subgroups and chain isometries. 
If $G$ has a chain $\C$ of subgroups, in Definition 7.1 we introduce the associated chain metric $d_\C$. In Remark 7.2 we show how the $q$-adic metric and the $RT$-metric can be naturally considered as chain metrics. In Definition 7.5 we define the notion of chain isometry, that is when two chains of subgroups of the same length of two groups of the same size are isometric.

To say that two groups are chain isometric gives more information than merely saying that they are isometric, since this implies that every step of the chains are isometric to each other (see \eqref{diagram}).
In Theorem 7.11, using geometric chains (see \eqref{cocientes}) we generalize Theorem~5.2 to groups not necessarily cyclic. More precisely, if $H < G$, with $|G|=q^n$ and $|H|=q$ then
$(G, d_\C) \simeq (H^{n}, d_{RT})$
where $d_\C$ is the chain metric associated to some chain of length $n$ with initial term $H$.
The most general result will be obtained in the next section.

Finally, in Section 8, we consider the block Rosenbloom--Tsfasman metric $d_{BRT}$ which generalizes the $RT$-metric 
(see Definition \ref{brt}). In Theorem 8.2 we prove that given a proper subgroup $H$ of a group $G$ and a chain $\C$ with initial term $H$ we have that $G$ with the metric $d_\C$ induced by the chain is isometric to $H^{[G:H]}$ with the block $RT$-metric, i.e.\@ 
$$(G,d_\C) \simeq (H^{[G:H]}, d_{BRT}).$$

\section{Invariant metrics on groups}
In this paper $X$ will always denote a finite set and $G$ a finite group. 
We begin by recalling some standard definitions.
A function $d :X \times X \rightarrow \mathbb{R}_{\geq 0}$ is called a \textit{metric} on $X$ if it is definite positive, symmetric, and satisfies the triangle inequality. That is, 
($a$) $d(x,y)\geq 0$ and $d(x,y)=0 \, \Leftrightarrow\, x=y$,
($b$) $d(x,y)=d(y,x)$,
and ($c$) $d(x,y)\leq d(x,z)+ d(z,y)$ hold for all $x,y,z\in X$. 
The pair $(X, d)$ is called a \textit{metric space}. 
If $d$ takes values in $\N_0$ and $\text{Im}(d) \subset [\![0,n]\!]$ with $n=|X|$ we say that $d$ is \textit{integral} and that $(X,d)$ is an \textit{integral metric space}.

Given an injective function $f : X \to Y$ between sets and a metric $d$ on $Y$, one can define the \textit{pullback metric} of $f$ on $X$ by 
\begin{equation} \label{fd*}
d_f(x,x') = d(f(x), f(x')),  \qquad x,x' \in X.
\end{equation} 

Two metric spaces $(X_1,d_1)$ and $(X_2,d_2)$ are said to be \textit{isometric},  
denoted by $(X_1,d_1) \simeq (X_2,d_2)$, 
if there is an \textit{isometry} between $X_1$ and $X_2$. 
That is, there is a bijection $\varphi : X_1 \rightarrow X_2$ such that for every $x,y \in X_1$ we have  
\begin{equation} \label{iso} 
d_1(x,y) = d_2(\varphi(x), \varphi(y)). 
\end{equation}
In other words, $d_1$ is the pullback metric of $d_2$.

A metric $d$ on $X$ can be naturally extended to the metric $d^n$ on $X^n$ in the following way
\begin{equation} \label{dn}
d^n(x,y) = \sum_{i=1}^{n} d(x_i,y_i)
\end{equation}
with $x=(x_1,\ldots,x_n), y=(y_1,\ldots,y_n) \in X^n$.
For instance, the Hamming metric $d_{\Ham}$ on $X$ extends to $X^n$ giving the most popular metric in coding theory 
$$d_{\Ham}^n(x,y) = \sum_{i=1}^{n}d_{\Ham}(x_i,y_i) = |\{ 1\leq i\leq n \;:\; x_i \neq y_i\}|.$$
Sometimes, the extended metric $d^n$ is also called $d$. 
This is a particular case of the product metric. If $(X_i,d_i)$, $i=1,\ldots,n$ are metric spaces then the \textit{product metri}c $d_\pi = d_1 \times \cdots \times d_n$ on $X=X_1 \times \cdots \times X_n$ is given by 
$$d_\pi(x,y) = d_1(x_1,y_1) + \cdots + d_n(x_n,y_n)$$
for $x=(x_1,\ldots,x_n)$, $y=(y_1,\ldots,y_n)$.

\sk 

A map $w : X \rightarrow \mathbb{R}$ is called a \textit{weight function} if
 $w(x) \geq 0$ for all $x\in X$ and $w(x)=0$ for exactly one element $x$ of $X$.
If $w$ takes integral values and moreover ${\rm Im}(w) \subset [\![0,N]\!]$ for some $N\in \mathbb{N}$ 
we will say that $w$ is an \textit{integral} weight. 
The pair $(X,w)$ is called a \textit{weight space} or integral weight space if $w$ is integral. 
If $(X,+)$ is a group, then $w$ must also satisfy the subadditive property, that is
$w(x+y) \le w(x)+w(y)$ for every $x,y \in X$.
Given a metric space $(X,d)$ and $a \in X$ we can canonically define a weight function $w_a$ by 
$$ w_a(x) = d(x,a), \quad x \in X.$$
If $|X|=n$, there are $n$ different weight functions as above.
For instance, if $X$ is a finite set and $x_0$ is a fixed element, the \emph{Hamming weight relative to $x_0$} is given by
\begin{equation}
    w_{x_0}(x)=d_{\Ham}(x,x_0)=
    \begin{cases}
      1 & \quad \text{if } x \neq x_0 , \\
      0 & \quad \text{if } x=x_0 .
    \end{cases}
  \end{equation}

If $(X,d)$ is a metric space with integral weight function $w$, the \textit{weight distribution} of $(X,d)$ is the set of weight frequencies $\{ A_0,A_1,\ldots,A_N\}$ where 
 $A_{i} = \#\{x\in X : w(x) = i\}$. The \textit{weight enumerator polynomial} of $(X,d)$ is defined by
 \begin{equation} \label{Wxd}
 \mathcal{W}_{(X,d)} (t) = \sum_{x\in X} t^{w(x)} = \sum_{i=0}^N A_i t^i.
 \end{equation}
Let $(X_1,d_1)$, $(X_2,d_2)$ be two metric spaces such that $0 \in X_1, X_2$ and consider the product space $X=X_1 \times X_2$ with the product metric $d_1 \times d_2$. Notice that we get 
$$\mathcal{W}_{(X,d)}(t) = \mathcal{W}_{(X_1,d_1)}(t) + \mathcal{W}_{(X_2,d_2)}(t) + \sum_{(x_1,x_2) \in X_1^* \times X_2^*} t^{w((x_1,x_2))}$$
where $X_i^*$ denotes $X_i \smallsetminus \{0\}$ for $i=1,2$.

All metrics and weights considered in this paper will be integral.

\subsection*{$G$-invariant metrics}
We are interested in the particular case in which $X=G$ is a group. 
The metric $d$ is called \textit{right (resp.\@ left) translation invariant} if for any $g,g',h$ in $G$ we have 
$$d(gh,g'h)=d(g,g')$$ 
(resp.\@ $d(hg,hg')=d(g,g')$). If $G$ is abelian both notions coincide and $d$ is called \textit{translation invariant}.   
There is a distinguished weight function $w(x)=d(x,e)$, where $e$ is the identity element of $G$. 
Also, if $(G,w)$ is a weight space, one can define a metric $d$ on $G$ by  
$$d(x,y) = w(x-y)$$ 
for every $x,y \in G$, provided that $w(-x)=w(x)$ holds for every $x \in G$ (or, in multiplicative notation, requiring that $d(x,y)=w(xy^{-1})$ with $w(x^{-1})=w(x)$ for every $x \in G$),  which is automatic for elementary $2$-groups.

Let $\Sym_{X}$ denote the permutation group of $X$. If $G$ acts on $X$ we have $G \le \Sym_X$.
\begin{defi}
Let $(X,d)$ be a metric space and $G \leq \Sym_X$.
We say that $(X,d)$ is \textit{$\sigma$-invariant} for $\sigma \in G$, denoted $d^\sigma=d$, if 
$$ d(\sigma (x),\sigma (y)) = d(x,y)$$ 
for all $x, y \in X$.
Further, $(X,d)$ is called \textit{$G$-invariant} if $d$ is \textit{$\sigma$-invariant} for every $\sigma \in G$.
The \textit{symmetry group} of $(X,d)$ is defined by
\begin{equation} \label{sym group}
\Gamma(X,d) = \{\sigma \in \Sym_{X}  :  d^\sigma =d \}.
\end{equation} 
We will say that $(X,d)$ has a \textit{$G$-representation} if there is a group $G \le \Gamma(X,d)$ which is \textit{regular} 
(or \textit{simply transitive}); that is, 
$|G|=|X|$ and the action of $G$ is transitive. 
\end{defi}

Notice that if $X=G$ and $d$ is right translation invariant then $d$ is $G_R$-invariant, where $G_R:G\rightarrow \Sym_G$
is the right regular representation given by $g \mapsto R_g$ for $g\in G$ with $R_g(x) = xg$ for any $x\in G$. 
Furthermore, $(X,d)$ is $G_R$-invariant if and only if $G_R \le \Gamma(X,d)$.
Similarly, the above facts hold for $d$ a left translation invariant metric and the left regular representation $G_L$.

\begin{rem}
Let $f:X \rightarrow Y$ be a bijective map from $X$ to a $G$-invariant metric space $(Y,d)$.
Then, the action $\sigma_Y$ of $G$ on $Y$ can be transferred to $X$ in such a way that the pullback metric $d_f$ becomes $G$-invariant. 
In fact, defining the action of $G$ on $X$ by $\sigma_X = f^{-1} \circ \sigma_Y \circ f$ we have
\begin{eqnarray*}
d_f(\sigma_X(x), \sigma_X(x')) &=& d_f \big( f^{-1}(\sigma_Y(f(x))), f^{-1}(\sigma_Y(f(x'))) \big) \\ 
&=& d \big( \sigma_Y(f(x)), \sigma_Y(f(x')) \big) = d(f(x),f(x')) = d_f(x,x')
\end{eqnarray*}
for $x,x'\in X$, where we have used that $d$ is $G$-invariant.
\end{rem}

\begin{ej}
Let $(X,d)$ be a metric space with $|X|=n$ and $d$ the discrete metric, that is $d(x,x)=0$ and $d(x,y)=1$ for every $x\ne y$.
Then $\Gamma (X,d) \simeq \Sym_n$ and, hence, $(X, d)$ has a $G$-representation for every group of order $n$, as a consequence of Cayley's Theorem.
\end{ej}

We now show that given a $G$-representation on a metric space $(X,d)$, the group $G$ inherits a metric and the set $X$ inherits a group structure.

\begin{prop} \label{transfer}
Suppose that the metric space $(X,d)$ has a $G$-representation. Then, there is a bijection $\varphi : X \rightarrow G$ which induces 
a group structure on $X$ and a metric $d_G$ on $G$ such that $\varphi$ is a group isomorphism and 
$\varphi: (X,d) \rightarrow (G,d_G)$ 
is an isometry.
Moreover, $d_G =d_{\varphi^{-1}}$ is translation invariant, that is 
$d_G(g_1 h, g_2 h) = d_G(g_1, g_2)$ for every $g_1, g_2, h\in G$. 
\end{prop}

\begin{proof}
Fix an element $x_0 \in X$.
Since $G$ acts regularly on $X$, $G$ acts transitively on $X$ and $|G|=|X|$. 
Thus, for each $x\in X$ there is a unique $g=g_x \in G$ such that $g(x_0) = x$.
Hence, we can define the map 
$$ \varphi:X \rightarrow G, \qquad  x \mapsto g_x.$$ 
This gives a group structure on $X$ by considering the product
$$xy = g_y(x).$$
To check associativity, note that $x(yz)=g_{yz}(x)$ and $(xy)z = g_{z}(xy)=g_zg_y(x)$. Since $g_zg_y(x_0) = zy = g_{yz}(x_0)$ we have that $g_zg_y = g_{yz}$ and hence $x(yz)=(xy)z$ for any $x,y,z \in G$.
It is easy to see that $x_0$ is the identity element in $X$ and 
$$x^{-1}=\varphi^{-1}(g_x^{-1})=g_x^{-1}(x_0).$$ 
Therefore, $\varphi$ is a group homomorphism and hence an isomorphism. 

Now, $d$ induces the metric $d_G$ in $G$ by 
$$d_G(g_x,g_y)=d(x,y).$$ 
Clearly, $(X,d)$ and $(G,d_G)$ are isometric since
$d_G(\varphi(x),\varphi(y))=d(x,y)$, by definition.
It only remains to show that $d_G$ is translation invariant. 
For $g_x,g_y,h \in G$ we have
\begin{eqnarray*}
d_G(g_x \cdot h, g_y \cdot h) &=& d_G((g_x \cdot h)(x_0), (g_y \cdot h)(x_0)) = d(h(g_x(x_0)), h(g_y(x_0))) \\ 
							  &=& d(h(x), h(y)) = d(x, y) = d_G(g_x, g_y)
\end{eqnarray*}
as we wanted to see. 
\end{proof}

\begin{rem}
A similar idea as in the previous proposition was established by Forney in (\cite{Forney}) in the context of geometrically uniform signal sets in $\mathbb{R}^n$ with the Euclidean metric.  
\end{rem}

We now illustrate the above proposition, showing that the $G$-representations of a set strongly depend on the chosen metric and on 
the symmetry of the group. For clarity we will sometimes use the graph of distances of a finite metric space $(X,d)$. If $|X|=n$, the graph of distances of $X$ is the weighted complete graph $K_n$ where each edge $xy$ has weight $d(x,y)$.

\begin{ej}
Consider the set $X=\{x,y,w,z\}$. Since $X$ has 4 elements, any $G$-representation of $X$ has only two possibilities: 
$G \simeq \Z_4$ or $G\simeq \Z_2 \times \Z_2$.

Consider some metrics on $X$ given by the following graphs of distances:

\begin{figure}[H]
\centering
\begin{minipage}[t][][t]{0.30\textwidth}
\begin{tikzpicture}[scale=0.75,transform shape,>=stealth',shorten
>=1pt,auto,node distance=3cm, thick,main node/.style={circle,draw,font=\Large,minimum size=8mm}]
\tikzstyle{every node}=[node distance = 2.8cm,bend angle    = 45,fill          =
gray!30]

  \node[main node,font=\fontsize{14}{144}] (1) {$x$};
  \node[main node,font=\fontsize{14}{144}] (2) [above left of=1] {$y$};
  \node[main node,font=\fontsize{14}{144}] (3) [above right of=2] {$z$} ;
  \node[main node,font=\fontsize{14}{144}] (4) [above right of=1] {$w$};

  \path[every node/.style={font=\large,circle}]
  
    (1) edge node[right] {1} (4)
       edge[red] node [right, pos=0.8] {2} (3)
    (2) edge node [left] {1} (1)
        edge[red] node [below, pos=0.3] {2} (4)
        
    (3) edge node [left] {1} (2)
        
    (4) edge node [right] {1} (3);

\end{tikzpicture}
\centering
\caption{$d_1$ }
\label{1}
\end{minipage}
\hfill
\begin{minipage}[t][][t]{0.30\textwidth}

\begin{tikzpicture}[scale=0.75,transform shape,>=stealth',shorten
>=1pt,auto,node distance=3cm, thick,main node/.style={circle,draw,font=\small,minimum size=8mm}]
\tikzstyle{every node}=[node distance = 2.8cm,bend angle    = 45,fill          =
gray!30]

  \node[main node,font=\fontsize{14}{144}] (1) {$x$};
  \node[main node,font=\fontsize{14}{144}] (2) [above left of=1] {$y$};
  \node[main node,font=\fontsize{14}{144}] (3) [above right of=2] {$z$} ;
  \node[main node,font=\fontsize{14}{144}] (4) [above right of=1] {$w$};

  \path[every node/.style={font=\large,circle}]

    (1) edge node[right] {1} (4)
       edge node [right, pos=0.8] {1} (3)
    (2) edge[red] node [left] {2} (1)
        edge node [below, pos=0.3] {1} (4)
        
    (3) edge node [left] {1} (2)
        
    (4) edge[red]  node [right] {2} (3);

\end{tikzpicture}
\centering
\caption{$d_2$ }
\label{2}
\end{minipage}
\hfill
\begin{minipage}[t][][t]{0.30\textwidth}

\begin{tikzpicture}[scale=0.75,transform shape,>=stealth',shorten
>=1pt,auto,node distance=3cm, thick,main node/.style={circle,draw,font=\small,minimum size=8mm}]
\tikzstyle{every node}=[node distance = 2.8cm,bend angle    = 45,fill          =
gray!30]

  \node[main node,font=\fontsize{14}{144}] (1) {$x$};
  \node[main node,font=\fontsize{14}{144}] (2) [above left of=1] {$y$};
  \node[main node,font=\fontsize{14}{144}] (3) [above right of=2] {$z$} ;
  \node[main node,font=\fontsize{14}{144}] (4) [above right of=1] {$w$};

  \path[every node/.style={font=\large,circle}]

    (1) edge [blue] node[right] {3} (4)
       edge[red] node [right, pos=0.8] {2} (3)
    (2) edge node [left] {1} (1)
        edge[red] node [below, pos=0.3] {2} (4)
        
    (3) edge [blue] node [left] {3} (2) 
        
    (4) edge  node [right] {1} (3);

\end{tikzpicture}
\centering
\caption{$d_3$ }
\label{3}
\end{minipage}
\hfill
\end{figure}

We now show that the number of $G$-representations of $(X,d)$ depends strongly on the chosen metric. 
Let 
$$G_1=\langle \rho= (xywz) \rangle \qquad \text{and} \qquad G_2 = \langle \tau_1=(xy)(wz), \tau_2=(xz)(yw) \rangle$$
be the groups defined by the permutations $\rho$ and $\tau_1,\tau_2$, respectively. Note that $G_1 \simeq \Z_4$,  $G_2 \simeq \Z_2 \times \Z_2$, and that they act transitively on $X$. 

\noindent ($i$) 
We have that $(X,d_1)$ is $G_i$-invariant, that is $G_i \le \Gamma(X,d_1)$, for $i=1,2$.
Fix $x$ as the identity element in $X$ and define $\varphi_1 : X \rightarrow G_1$ as follows
$$x \mapsto e, \quad y \mapsto \rho=(xywz), \quad w \mapsto \rho^2=(xw)(yz), \quad z \mapsto \rho^3=(xzwy).$$ 
Also, define $\varphi_2 : X \rightarrow G_2$ by 
$$x \mapsto e, \quad y \mapsto \tau_1=(xy)(wz), \quad w \mapsto \tau_1\tau_2=(xw)(yz), \quad z \mapsto \tau_2=(xz)(yw).$$ 
By Proposition \ref{transfer}, there are isometries $(X,d_1) \simeq (G_1,d_{G_1})$ and $(X,d_1) \simeq (G_2, d_{G_2})$.
Note that under the isomorphisms $G_1 \simeq \Z_4$ and $G_2 \simeq \Z_2 \times \Z_2$ the metrics $d_{G_1}$ and $d_{G_2}$ correspond to the Lee metric $d_{Lee}$ on $\Z_4$ and to the Hamming metric $d_{\Ham}^2$ on $\Z_2 \times \Z_2$. That is 
$$(X,d_1) \simeq (\Z_4,d_{Lee}) \qquad \text{and} \qquad (X,d_1) \simeq (\Z_2 \times \Z_2,d_{\Ham}^2).$$ 
In particular, by transitivity, we have recovered the known isometry between $\Z_4$ and $\Z_2 \times \Z_2$ given by the Gray map.

\noindent ($ii$) 
Consider now the metric $d_2$. Notice that $(X,d_2)$ is $G_2$-invariant but it is not $G_1$-invariant. In this case, 
$(X,d_2)$ has only one $G$-representation with $G\simeq \Z_2 \times \Z_2$.

\noindent ($iii$) 
Finally, observe that when the metric $d_3$ is considered, the metric space $(X,d_3)$ has no $G$-representations at all because the group of symmetries of $(X,d)$ is trivial (none of the groups can preserve the distances). 
\hfill $\lozenge$
\end{ej}

\begin{rem} \label{rem inv}
From now on, if $G$ is a group, $(G,d)$ will denote a metric space where the distance $d$
is $G$-invariant and $G$ acts by right translations, i.e.\@ we identify $G$ with its right regular representation $G_{R} \leq \Sym_G$.
\end{rem}

\section{Hamming spaces are not isometric to cyclic groups}
Due to the relevance shown by the Gray map $\mathcal{G} : \mathbb{Z}_4 \rightarrow \Z_2^2$ in coding theory, 
people was concerned whether there is a generalization of this isometry
sending $\Z_{p^{n}}$ to $(\Z_p)^n$, with $p$ prime.
As already mentioned in the Introduction, Salagean-Mandache proved (\cite{Salagean-Mandache99}) that, except for the known case $p=n=2$, 
it is impossible to construct a metric $d$ in $\Z_{p^{n}}$ such that $(\Z_{p^{n}},d)$ is isometric to 
$(\Z_p^n,d_{\Ham})$. 
Sueli Costa and collaborators deal with the existence of isometries of a Hamming space $X^n$, 
$|X|=m$ (see \cite{Symmetry-Lee} for $m=p$ prime, \cite{DBLP:journals/dm/AlvesC03} for arbitrary $m$). 
They proved that there are no $G$-representations of the Hamming space $X^n$ with $G$ a cyclic group, 
except for the Gray map and the trivial case $n=1$; that is, we have

\begin{thm}[\cite{DBLP:journals/dm/AlvesC03}]	\label{IsoCiclicoNo}
Let $(X^n,d_{\Ham})$ be a Hamming space, with $|X|=m$. If $(m,n)\neq (2,2)$ and $n>1$, there does not exist any cyclic group $G$ and any metric $d$ on $G$ such that $(G,d)$ is isometric with $(X^n,d_{\Ham})$. 
\end{thm}

We will give an alternative simple proof of this result, using group theory, in the case that $|X|=p$ is prime.
 
\begin{lema} \label{SubgroupsIso} 
If $G$ is a finite group containing two subgroups $H \simeq \mathbb Z_p^k$ and $K \simeq \mathbb Z_{p^\ell}$, 
with $p$ prime and $k,\ell \in \N$, 
then the order of $G$ is divisible by $p^{k+\ell-1}$.
\end{lema}

\begin{proof}
By Sylow's theorems it is enough to consider only the case when $G=P$ is a $p$-group.
Therefore, $P$  contains subgroups $H$ and $K$ isomorphic to $\mathbb Z_p^k$ and $\mathbb Z_{p^\ell}$ respectively. Then we have that
$$|P| \geq |HK| = \frac{|H| |K|}{|H\cap K|}\geq \frac{p^k \cdot p^\ell}{p} = p^{k+\ell-1},$$
where we have used that $|H\cap K|=1$ or $p$, since $H\cap K$ is cyclic of order $p$.
Since $|P|$ is a power of $p$ and $|P| \ge p^{k+\ell-1}$, then $|P|$
is divisible by $p^{k+\ell-1}$. 
\end{proof}

We now restate Theorem \ref{IsoCiclicoNo} in terms of representations for spaces of prime cardinality. 

\begin{prop} \label{prop isometrias}
Let $(X^n,d_{\Ham})$ be a Hamming space, with $|X|=p$ prime. If $(p,n)\neq (2,2)$ and $n>1$, then there is no cyclic 
representation of $(X^n,d_{\Ham})$. In particular,
there is no isometry between $\mathbb Z_{p^{n}}$ and $(\mathbb Z_p^{n},d_{\Ham})$.
\end{prop}

\begin{proof}
The Hamming space $X^n$ has a cyclic representation if and only if the symmetry group has an element of order $p^{n}$, such that the subgroup generated by this element acts regularly. 
It is known that the symmetry group of the Hamming space is (see for instance \cite{Cam} or \cite{PP})
$$\Gamma (X^n,d_{\Ham})\simeq \Sym_p \wr \Sym_n = (\Sym_p)^n \rtimes \Sym_n,$$ 
where $\wr$ denotes wreath product. 
Assume that there exists a cyclic representation of $(X^n,d_{\Ham})$. We have that $\mathbb Z_{p^{n}} \subsetneq \Gamma(X^n,d_{\Ham})$ 
and also that $\mathbb{Z}_p^{n} \subsetneq \Gamma (X^n,d_{\Ham})$. Thus, by Lemma \ref{SubgroupsIso}, 
$p^{n+n-1}$ must divide $|\Gamma(X^n,d_{\Ham})|$, that is 
$$ p^{2n-1} \mid (p!)^{n}n!$$
On the other hand, note that if $\nu_p$ denotes the $p$-adic valuation, we have 
$$ \nu_p((p!)^{n}n!) = n \nu_p(p!) + \nu_p(n!) = n \nu_p(p) + \nu_p(n!) = n+\nu_p(n!). $$
Now, suppose that $n=n_0 + n_1 p + n_2p^{2} + \cdots + n_rp^{r}$ 
is the $p$-adic expansion of $n$, and let $s_p(n) = n_0 +n_1 + \cdots+n_r$. Then, the Legendre formula for the $p$-adic valuation of $n!$ implies that $\nu_p(n!)=\frac{n-s_p(n)}{p-1}$. 
Then we have that
\begin{equation} \label{desigualdad}
\nu_p((p!)^{n}n!) = n + \tfrac{n-s_p(n)}{p-1} \leq n+ n-1 = 2n-1.
\end{equation} 
 Moreover, the equality holds in \eqref{desigualdad} if and only if $p=2$ and
 $n=2^{k}$ for some $k$.
 
It only remains to prove the case $p=2$. 
It is enough to show that if 
$$g\in \Gamma (X^n,d_{\Ham}^{n}) \simeq \mathbb Z_2^{n} \rtimes \Sym_n$$ 
then its order satisfies $|g|<2^n$. 
We recall Landau's function $G(n)=\max \{ord(\sigma) : \sigma \in \Sym_n\}$ and the known bound $G(n) \le e^{\frac ne}$. 
Now, if $g=(t,s)$ with $t\in \mathbb Z_2^{n}$ and $s\in \Sym_n$ then we have 
 $|g|\leq |t||s|\leq 2e^{\frac{n}{e}}$, by the bound on Landau's function. 
In particular if $n> 2$,  
$$|g|\leq 2e^{\frac{n}{e}}<2^n,$$ 
 and hence there is no element of order $2^{n}$ in $\Gamma (X^n,d_{\Ham}^{n})$.
\end{proof}

\section{Isometric embeddings}\label{Isometric embeddings}
We begin with the following definition.
\begin{defi}
A map $\varphi : (X_1,d_1) \rightarrow (X_2,d_2)$ between metric spaces is an \textit{isometric embedding}
if it is injective and preserves distances.
That is, for every $x,y \in X_1$ we have 
\begin{equation*} 
d_1(x,y) = d_2(\varphi(x), \varphi(y)). 
\end{equation*}
\end{defi}

As we previously mentioned, for any fixed $m$, the cyclic group $\Z_m$ cannot be isometric to any Hamming space $(X^n,d_{\Ham}^n)$ 
where $m=|X^n|$. However, there are isometric embeddings of $\Z_{p^k}$ into the Hamming space $\Z_{p}^{p^{k-1}}$ with $p$ prime due to Carlet (\cite{Ca}, $p=2$), Greferath and Yildiz--\"Ozger (\cite{Greferath} and \cite{YO}, any prime), thus generalizing the Gray map. Namely, we have the following result. 
 
\begin{thm}[\cite{Ca}, \cite{Greferath}, \cite{YO}] \label{IsoSi}
Let $p$ be a prime and $k > 1$, then there exists an isometric embedding from $(\mathbb Z_{p^{k}}, d)$ to 
$(\mathbb \Z_{p}^{p^{k-1}}, d_{\Ham}^{p^{k-1}})$.
\end{thm}

In this section we will generalize the previous result to any cyclic group. More precisely, we will see that, for any $m \in \N$, it is always possible to isometrically embed $\Z_m$ into a Hamming space
 $X^n$ with $m < |X^n|$ for some $n$. For simplicity we will denote this by
$$\Z_m \hookrightarrow (X^n,d_{\Ham}^n).$$

\sk 
Let $G=\Z_m=\{0,1,\ldots,m-1\}$ and $H$ be a subgroup of index $n$, i.e.\@ $n=[G:H]$.
Consider $v\in H^n$ and $\rho \in \Sym_n$.
We define the map (well-defined since $H$ is a group)
$$\Psi_{v,\rho} : \Z_m \rightarrow H^n, \qquad t \mapsto \Psi_t(\rho) \cdot v$$ 
where $\rho \cdot v$ is the action $\rho(v_1,\ldots,v_n ) = (v_{\rho(1)},\ldots,v_{\rho(n)})$ and 
$$\Psi_t(x) = \tfrac{x^t-1}{x-1} = x^{t-1} + \cdots + x+1.$$  

In the previous notations, we have the following.
\begin{lema} \label{1-1}
Let $H$ be a subgroup of $G=\Z_m$ of index $n$.
Suppose $H=\langle h \rangle$ and $e_i=(0,\ldots,1,\ldots,0) \in \Z_m^n$ with $1$ in the $i$-th coordinate. 
If $v=he_i$ and $\rho$ is an $n$-cycle then $\Psi_{v,\rho}$ is injective. 
\end{lema}

\begin{proof}
For $0\leq t <m$, if $t=qn +r$ with $0\leq r < n$, then using that $\rho(e_i) = e_{\rho(i)}$ we have
\[
\Psi_{v,\rho}(t) = \sum_{k=0}^{t-1} \rho^k(he_i)= (q+1) \sum_{k=0}^{r-1} h \, e_{\rho^{k}(i)} + q \sum_{k=r}^{n-1} h \, e_{\rho^{k} (i)}.
\]
Now, for $0\leq s < m$, if $s=q'n +r'$ with $0\leq r' < n$, we can see that $\Psi_{v,\rho}(s)= \Psi_{v,\rho}(t)$ if and only if $q=q'$, and $r=r'$, that is, only if $s=t$. Therefore $\Psi_{v,\rho}$ is injective.
\end{proof}

We now generalize Theorem \ref{IsoSi} to $\Z_m$, with $m$ any positive integer.
\begin{thm} \label{teo emb}
Let $H$ be a subgroup of $G=\Z_m$ of index $n$.
Consider $v=he_i \in H^n$ with $H=\langle h \rangle$, $1\le i \le n$, and $\rho \in \Sym_n$ an $n$-cycle.
Then we have the isometric embedding 
\begin{equation} \label{isom embedding}
\Psi_{v,\rho} : (\Z_m, d_n) \hookrightarrow (H^n, d_{\Ham}^n),
\end{equation}
where $d_n$ is the translation invariant metric 
with associated weight given by
\begin{equation} \label{weightexplicit}
w_{n}(t) =
\begin{cases}
t  & \quad \text{if } \; t \leq n, \\
n  & \quad \text{if } \; n \leq t \leq m-n, \\
m-t & \quad \text{if } \;  m-n \leq t \leq m-1.\\
\end{cases}
\end{equation}
\end{thm}

\begin{proof}
Consider $\tilde{\Psi}_{v,\rho} : \Z_m \rightarrow H^n\rtimes \Sym_n$ given by $t \mapsto (\Psi_t(\rho) \cdot v, \rho^t)$ and notice that it is a homomorphism. In fact, given $t,s \in \Z_m$ we  have 
\begin{eqnarray} \label{psi}
\tilde{\Psi}_{v,\rho}(t) \tilde{\Psi}_{v,\rho}(s) &=& ( \Psi_t(\rho)v, \rho^t ) \, (\Psi_s(\rho)v,\rho^s) \nonumber  \\
&=& (\rho^s \Psi_t(\rho)v + \Psi_s(\rho)v, \rho^t\rho^s) \nonumber  \\
&=& \big( \rho^s (\rho^{t-1} + \cdots + \rho + 1)v + (\rho^{s-1} + \cdots + \rho + 1)v, \rho^{t+s} \big) \\
&=& \big( (\rho^{t+s-1} + \cdots + \rho+1)v,\rho^{t+s} \big) \nonumber \\ 
&=& (\Psi_{t+s}(\rho)v, \rho^{t+s}) = \tilde{\Psi}_{v,\rho}(t+s). \nonumber 
\end{eqnarray}
Further, if $t+s\equiv u \pmod m$, then $\Psi_{v,\rho}(t+s)=\Psi_{v,\rho}(u)$ and $\rho^{t+s}=\rho^{u}$, and hence we get $\tilde{\Psi}_{v,\rho}(s+t)=\tilde{\Psi}_{v,\rho}(u)$.

\sk  
Note that $\Psi_{v,\rho} = \pi \circ \tilde{\Psi}_{v,\rho}$, so we have the following diagram
\begin{equation} \label{diagram}
\begin{tikzcd}
\Z_m \arrow[r, "\tilde{\Psi}_{v,\rho}"] \arrow[dr, "\Psi_{v,\rho}"']
& H^n \rtimes \Sym_n  \arrow[d,"\pi"]\\
& H^{n} .
\end{tikzcd}
\end{equation}
Now, $\Psi_{v,\rho}$ is 1-1 by Lemma \ref{1-1} and hence $\tilde\Psi_{v,\rho}$ is also injective.

\sk 
Denote $\Psi_{v,\rho}$ by $\Psi$ and let $d_n$ be the pull-back metric in $\Z_m$ of the Hamming metric in $H^n$, that is
$$d_n(a,b) = d_{\Ham}^n(\Psi(a),\Psi(b)).$$
Hence $\Psi$ preserves the metric $d_n$ by definition.

\sk 
We now prove that $d_n$ is translation invariant. 
Note that $H^n \rtimes \Sym_n \subsetneq \Sym_H^n \rtimes \Sym_n$, since $H \subsetneq \Sym_H$ by Cayley's Theorem, and 
$\Gamma(H^n,d_{\Ham}^n) \simeq \Sym_H^n \rtimes \Sym_n$. Thus, we have 
\begin{equation} \label{cay}
H^n \rtimes \Sym_n \subsetneq \Gamma(H^n,d_{\Ham}^n). 
\end{equation} 
In this way, for every $a,b,c \in \Z_m$ we have
\begin{eqnarray*}
d_n(a+c,b+c) &=& d_{\Ham}^n(\Psi(a+c),\Psi(b+c)) \\ 
&=& d_{\Ham}^n \big(\rho^{c} \Psi(a)+ \Psi(c) , \rho^{c} \Psi(b)+ \Psi(c) \big) \\ 
&=&  d_{\Ham}^n(\Psi(a), \Psi(b)) = d_n(a,b),
\end{eqnarray*}
where in the second equality we have used \eqref{cay} and that $\Psi(a+c)=\rho^{c} \Psi(a)+ \Psi(c)$, deduced from  \eqref{psi}.

\sk 
Finally, we check the weights.
For $t \in \Z_m$ we have 
$$w_n(t) = d_{\Ham}^n(\Psi(t),0) = w_{\Ham}(\Psi(t) ) = w_{\Ham}((\rho^{t-1}+\cdots + 1) \cdot he_i).$$
Thus, considering $t=qn+r$, with $0\leq r < n$, we arrive at
$$w_n(t) =  w_{\Ham} \big( (q+1) \sum_{k=0}^{r-1} h \, e_{\rho^{k}(i)} + q \sum_{k=r}^{n-1} h \, e_{\rho^{k} (i)} \big), $$
from which \eqref{weightexplicit} readily follows.
\end{proof}

In the situation of the previous theorem, there are $\phi(\tfrac mn) n!$ different isometric embeddings $\Psi_{he_i,\rho}$, 
where $\phi$ is the Euler totient function. Indeed, there are $n$ vectors $e_i$, $(n-1)!$ different $n$-cycles $\rho$ 
and $\phi(\tfrac mn)$ different generators $h$ of $H$ of index $n$ in $\Z_m$.
However, all these maps have the same associated metric $d_n$.

\begin{rem} \label{YO recover}
Let $G=\Z_{p^k}$ with $p$ prime and for $1\le i \le k-1$ consider the subgroup $H_i=\Z_{p^i}$ of index $n_i=p^{k-i}$. 
By the previous theorem, there is an isometric embedding 
$$(\Z_{p^k}, d_{n_i}) \hookrightarrow \big( (\Z_{p^i})^{p^{k-i}}, d_{\Ham}^{\, p^{k-i}} \big)$$
determined by $\Psi_{e_j,\rho}$ for any $e_j \in H_i^{n_i}$ and any $n_i$-cycle $\rho$ in $\Sym_{n_i}$.
In particular, if we take $H_1=\Z_p$ , $v=e_1=(1,0,\ldots,0)$ and $\rho = (12\cdots n_1)$ then the weight $w_{n_1}$ becomes the extended Lee weight over $\Z_{p^k}$ 
$$ w_{_L}(x) = \begin{cases}
x 		 & \quad \text{if } \; x \leq p^{k-1}, \\
p^{k-1}  & \quad \text{if } \; p^{k-1} \le x \le p^{k}-p^{k-1}, \\
p^{k}-x  & \quad \text{if } \; p^{k}-p^{k-1} \le x \leq p^{k}-p^{k}-1.\\
\end{cases}$$
and thus we recover the isometric embedding 
$(\Z_{p^k}, d_{_L}) \hookrightarrow (\Z_p^{p^{k-1}},d_{\Ham}^{\, p^{k-1}})$ previously given by Yildiz--\"Ozger (\cite{YO}, Theorem 2.1).
\end{rem}

\begin{rem}
Consider $G=\Z_m$. One could want $H^n$ to be of the least possible size such that $\Psi$ is close to be a bijective embedding 
(i.e.\@ an isometry). In this case, we must choose $H$ minimizing $|H|^{n}$. On the other hand, if we want to minimize the dimension of the Hamming space of the embedding, we should choose $H$ to be the subgroup of maximum cardinality.

For instance, let $G = \Z_{{p_1}^{k_1}{p_2}^{k_2}}= \Z_{{p_1}^{k_1}} \times \Z_{{p_2}^{k_2}}$ where $p_1 < p_2$ are different primes. 
By Theorem~\ref{teo emb} and choosing $H=\Z_{{p_1}^{k_1-1} {p_2}^{k_2}}$ in order to minimize the size of the embedding space, we have the isometric embedding 
$$\Z_{p_1^{k_1} p_2^{k_2}} \hookrightarrow \big( (\Z_{p_1^{k_1-1}  p_2^{k_2}})^{p_1}, d_{\Ham}^{\, p_1} \big) .$$
On the other hand, we can apply the same theorem to $\Z_{{p_1}^{k_1}}$ and $\Z_{{p_2}^{k_2}}$ separately and then concatenate the spaces obtaining the isometric embedding 
$$ \Z_{p_1}^{k_1} \times \Z_{p_2}^{k_2} \hookrightarrow 
\big( (\Z_{p_1^{k_1-1}})^{p_1} \times (\Z_{p_2^{k_2-1}})^{p_2}, d_{\Ham}^{\, p_1}\times d_{\Ham}^{\,p_2} \big ).$$ 
However, note that although in general the second group is smaller, we must pay the price that the coordinates have different alphabets.
\end{rem}

\begin{rem} \label{chars}
It is possible to isometrically embed a ring into rings of different characteristics.
In fact, if $G=\Z_{pq}$ with $p,q$ primes, by Theorem \ref{teo emb} we have $\Z_{pq} \hookrightarrow ((\Z_p)^q, d_{\Ham}^{\, q})$ and   
$\Z_{pq} \hookrightarrow ((\Z_q)^p, d_{\Ham}^{\,p})$.
\end{rem}

\begin{ej}
Let $G=\Z_{12}$, we can consider the four subgroups $H_1 \simeq \Z_2$, $H_2 \simeq \Z_3$, 
$H_3 \simeq \Z_4$ and $H_4 \simeq \Z_6$ with corresponding indices $n_1=6$, $n_2=4$, $n_3=3$ and $n_4=2$. 
Thus, by Theorem \ref{teo emb} we have 
the four isometric embeddings 
$\Z_{12} \hookrightarrow (H_i^{n_i}, d_{\Ham}^{\, n_i})$ for $i=1,2,3,4$. 
That is 
$$\Z_{12} \hookrightarrow (\Z_2^6, d_{\Ham}^{6}), \quad \Z_{12} \hookrightarrow (\Z_3^4, d_{\Ham}^{4}), 
\quad \Z_{12} \hookrightarrow (\Z_4^3, d_{\Ham}^{3}), \quad \Z_{12} \hookrightarrow (\Z_6^2, d_{\Ham}^{2}).$$
By \eqref{weightexplicit}, we have the following weight distributions 
$$\begin{tabular}{|c|cccccccccccc|}
\hline
 $t$     & 0 & 1 & 2 & 3 & 4 & 5 & 6 & 7 & 8 & 9 & 10 & 11 \\ \hline
$w_1(t)$ & 0 & 1 & 2 & 3 & 4 & 5 & 6 & 5 & 4 & 3 & 2  &  1 \\
$w_2(t)$ & 0 & 1 & 2 & 3 & 4 & 4 & 4 & 4 & 4 & 3 & 2  &  1 \\
$w_3(t)$ & 0 & 1 & 2 & 3 & 3 & 3 & 3 & 3 & 3 & 3 & 2  &  1 \\
$w_4(t)$ & 0 & 1 & 2 & 2 & 2 & 2 & 2 & 2 & 2 & 2 & 2  &  1 \\
\hline
\end{tabular}$$
The corresponding weight enumerators are 
\begin{equation}
\begin{split}
&\mathcal{W}_{(\Z_{12}, d_1)}(t) = t^6 + 2t^5 + 2t^4 + 2t^3 + 2t^2 + 2t +1 ,\\
&\mathcal{W}_{(\Z_{12}, d_2)}(t) = 5t^4 + 2t^3 + 2t^2 + 2t +1 ,\\
&\mathcal{W}_{(\Z_{12}, d_3)}(t) = 7t^3 + 2t^2 + 2t +1 ,\\
&\mathcal{W}_{(\Z_{12}, d_4)}(t) = 9t^2 + 2t +1 .
\end{split}
\end{equation}

Note that the associated metrics $d_i$ obtained are all different and that the metric $d_1$ is just the Lee metric.
\end{ej}

\begin{rem} \label{lee rec}
In general, considering different subgroups $H$ of $G$, the isometric embeddings provided by Theorem \ref{teo emb} give rise to different metrics. In the particular case that $G=\Z_{2n}$, we have $H=\Z_n$, and the isometric embedding 
$$\Z_{2n} \hookrightarrow ((\Z_2)^n, d_{\Ham}^{\, n})$$ 
recovers the Lee metric on $\Z_{2n}$ since the associated weight function $w$ is given by 
$$\big(w(i)  \big)_{i=0}^{2n-1} = (0,1,2,\ldots, n-1, n, n-1,\ldots 2,1).$$  
\end{rem}

\section{Isometries between $\Z_{q^n}$ and $\Z_q^n$}\label{Zqn}
Here we will prove that the groups $\Z_{q^n}$ and $\Z_q^n$ are isometric for positive integers $n$ and $q$ with $q\ge 2$ by using metrics different from the Hamming metric. 
Namely, the $RT$-metric in $\Z_q^n$ and the $q$-adic metric on $\Z_{q^n}$ that we now define.

\begin{defi} \label{def dq}
Let $n,q \in \N$ with $q \ge 2$. The $q$-\textit{adic} metric $d_q$ in $\Z_{q^n}$ is given by
\begin{equation}  \label{qadica}
  d_q(x,y) = \min_{0 \le i \le n} \{i : q^{n-i} \mid x-y\}
\end{equation}
for any $x,y \in \Z_{q^n}$.
\end{defi}

Indeed, $d_q$ is a translation invariant metric. To check that it is a metric it is enough to show the triangle inequality, the other conditions being straightforward.
Let $x,y,z \in \Z_{q^n}$ and suppose that 
$$i=d(x,z), \quad j=d(z,y) \quad \text{ and } \quad k=d(x,y).$$ 
Then, $q^{n-i} \mid x-z$ and $q^{n-j} \mid z-y$, and thus we have that  
$$q^{n-\max\{i,j\}} \mid (x-z)+(z-y)=x-y.$$ 
Hence we have $k \le \max\{i,j\}$ and therefore 
$d(x,y) \le d(x,z) + d(z,y)$. In fact, $d_q$ is an ultrametric.
Finally, we have $d_q(x+z,y+z) = d_q(x,y)$ by definition, hence $d_q$ is translation invariant.
Notice that alternatively we have
$$d_q(x,y)= \lceil \log_q \big( ord(x-y) \big) \rceil,$$
where $ord$ denotes the order of an element in the group.
In particular, if $q=p$ is prime we simply get $d_p(x,y)=\log_p \big( ord(x-y) \big)$.

We recall that the \textit{Rosenbloom--Tsfasman metric} (or $RT$-metric) was originally defined over $\F_q^{n}$ (\cite{RT-metric}), 
hence for $q$ a prime power. However, this metric can be defined over $G^n$ for any group $G$. Thus, we now define the \textit{$RT$-metric} on $\Z_q^{n}$ for any pair of integers $n, q$ with $q \ge 2$ as follows:
\begin{equation} \label{RT metric}
d_{RT}(x,y) = \max_{1\leq i\leq n} \left\{ i \;:\; x_i-y_i \ne 0\right\}.
\end{equation} Note that $d_{RT}$ is translation invariant by definition.
It is known that it coincides with the poset metric $d_P$ on $\Z_q^{n}$
given by the chain poset $P$ defined by $1\preceq 2 \preceq \cdots\preceq n$ (see \cite{PFS}).

Now, we construct an explicit isometry between the groups $\Z_{q^n}$ and $\Z_q^{n}$ with the previous metrics. 
Let $q\ge 2$ and $n$ be positive integers and consider the function
\begin{gather} \label{iso RT q}
\begin{gathered} 
 \varphi : \Z_q^n \rightarrow \Z_{q^n} \\[1mm]
 \varphi(a_1,a_2,\ldots,a_n) \mapsto
a_1q^{n-1}+a_2q^{n-2}+\cdots+a_{n-1}q+a_n \pmod{q^{n}}.
\end{gathered}
\end{gather}
One can check that its inverse 
\begin{equation} \label{iso inv} 
\varphi^{-1} : \Z_{q^n}\rightarrow \Z_q^{n}
\end{equation} 
is given by the $q$-base expansion, namely 
\begin{equation} \label{qadic} 
\begin{array}{ccl}
0 &\mapsto & 0000 \cdots 000 \\
1 &\mapsto & 0000 \cdots 001 \\
\vdots &\qquad    & \phantom{00000} \vdots\\
q-1 &\mapsto & 0000 \cdots 00(q-1)\\
q &\mapsto  & 0000 \cdots 010\\ 
q+1 &\mapsto & 0000 \cdots 011\\
\vdots &\qquad    & \phantom{00000} \vdots\\
q^{2}-1 &\mapsto & 0000 \cdots 0(q-1)(q-1)\\
q^{2} &\mapsto  & 0000 \cdots 100\\ 
q^{2}+1 &\mapsto & 0000 \cdots 101\\
\vdots &\qquad    & \phantom{00000} \vdots  \\
q^{n}-1 &\mapsto & (q-1)(q-1)(q-1)(q-1) \cdots (q-1)(q-1)(q-1)
\end{array}
\end{equation}

We now show that $\varphi$ as in \eqref{iso RT q} preserves distances.
\begin{thm} \label{IsoRT}
For any $n,q \in \N$ with $q\ge 2$ the map $\varphi : (\Z_q^n, d_{RT}) \rightarrow (\Z_{q^n}, d_q)$ as in \eqref{iso RT q} is an isometry.
\end{thm}

\begin{proof}
To see that $\varphi$ is an isometry between metric groups we must show that $\varphi$ preserves distances and that 
the involved metrics are translation invariant.

Let $x,y\in \Z_q^{n}$ and suppose that $d_{RT}(x,y)=k$. This means that
$x_k \neq y_k$ and $x_i = y_i$ for $i=k+1,\ldots,n$.
On the other hand, 
$$d_q(\varphi (x),\varphi(y)) = \min_{0 \le i \le n}\{i : q^{n-i} \mid \varphi(x) - \varphi(y) \}$$
where, by \eqref{iso RT q}, we have that
$$\varphi(x)-\varphi(y) = (x_1-y_1) \, q^{n-1} + (x_2-y_2) \, q^{n-2} + \cdots + (x_{k}-y_{k}) \, q^{n-k}
\pmod{q^{n}},$$
with $x_{k}-y_{k}\neq 0$. Thus  
$$d_q(\varphi (x),\varphi(y)) = k = d_{RT}(x,y)$$
and hence $\varphi$ preserves distances. Finally, we have previously observed that both $d_{RT}$ and $d_q$ are translation invariant and the result thus follows. 
\end{proof}

Notice that, by \eqref{qadica}, \eqref{RT metric} and Theorem \ref{IsoRT}, the weight enumerators are 
\begin{equation} \label{enum qadic}
\mathcal{W}_{(\Z_{q^n}, d_q)}(t) = \mathcal{W}_{(\Z_q^n, d_{RT})}(t) = \sum_{i=0}^n (q^i-q^{i-1}) \, t^i 
= (q-1) \sum_{i=0}^n q^{i-1} \, t^i. 
\end{equation}

\begin{exam}
We now illustrate the previous theorem showing that the groups $\Z_2^3$ and $\Z_8$ are isometric. 
We take the $d_{RT}$ metric  on $\Z_2^3$ and the $2$-adic metric $d_2$ on $\Z_8$.
In this case, the map $\varphi^{-1}:\Z_8 \rightarrow \Z_2^{3}$ in \eqref{iso inv} is given by 
\begin{align*}
0 \; & \mapsto (0,0,0), & 2 \; &\mapsto (0,1,0), & 4 \; &\mapsto (1,0,0), & 6 \; &\mapsto (1,1,0), \\
1 \; & \mapsto (0,0,1), & 3 \; &\mapsto (0,1,1), & 5 \; &\mapsto (1,0,1), & 7 \; &\mapsto (1,1,1).
\end{align*}
The graphs of distances of the groups are as follows: 

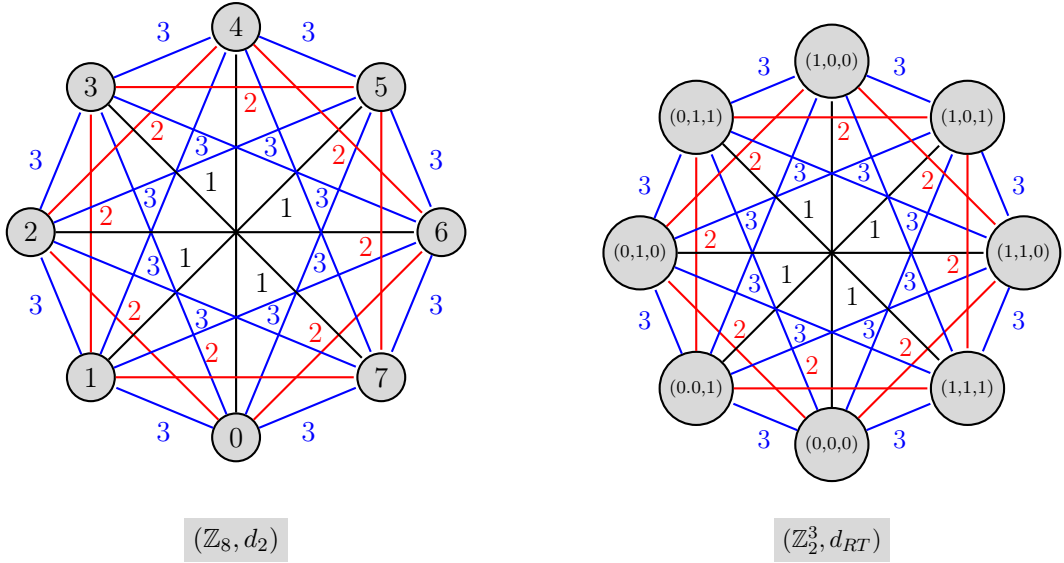
\begin{center}
\begin{minipage}[t][][t]{0.485\textwidth}
\begin{tikzpicture}[scale=0.91,transform shape,>=stealth',shorten
>=1pt,auto,node distance=4cm, thick,main node/.style={circle,draw,font=\large}]
\tikzstyle{every node}=[node distance = 3.5cm,bend angle    = 45,fill          =
gray!30]

  \node[main node] (0) at ({6*45}:3) {0};
  \node[main node] (1) at ({5*45}:3) {1};
  \node[main node] (2) at ({4*45}:3) {2} ;
  \node[main node] (3) at ({3*45}:3) {3};
  \node[main node] (4) at ({2*45}:3) {4};
  \node[main node] (5) at ({1*45}:3) {5} ;
  \node[main node] (6) at ({0*45}:3) {6};
  \node[main node] (7) at ({7*45}:3) {7};
  \node     at (0,-4.5) {$(\Z_8, d_2)$};

  \path[every node/.style={font=\large,circle}]

    (0) edge node[left,pos=0.64] {1} (4)
       edge[red] node [right, pos=0.64] {2} (2)
       edge[red] node [left, pos=0.51] {2} (6)
       edge[blue] node[below] {3} (1)
       edge[blue] node[left=-0.2cm, pos=0.49] {3} (3)
       edge[blue] node[right=-0.2cm] {3} (5)
       edge[blue] node[below] {3} (7)
    (1) edge node [left,pos=0.40] {1} (5)
        edge[red] node [right=-0.14cm, pos=0.57] {2} (3)
        edge[red] node [above, pos=0.4] {2} (7)
        edge[blue] node[left] {3} (2)
        edge[blue] node[left=-0.14cm, pos=0.516] {3} (4)
        edge[blue] node[right=0.1cm,pos=0.42] {3} (6)
        
    (2) edge node [above,pos=0.64] {1} (6)
      edge[red] node [below, pos=0.64] {2} (4)
      edge[blue] node[left] {3} (3)
      edge[blue] node[above=-0.14cm,pos=0.486] {3} (5)
      edge[blue] node[left,below=-0.14cm,pos=0.486] {3} (7)
        
    (3) edge node [left,pos=0.70] {1} (7)
      edge[red] node [below=-0.14cm, pos=0.56] {2} (5)
      edge[blue] node[above] {3} (4)
      edge[blue] node[above=-0.14cm,pos=0.516] {3} (6)
      
    (4) edge[red] node [below, pos=0.5] {2} (6)
       edge[blue] node[above] {3} (5)
       edge[blue] node[right=-0.14cm,pos=0.486] {3} (7)
    
    (5) edge[red] node [left=-0.14cm, pos=0.56] {2} (7)
    edge[blue] node[right] {3} (6)
    
    (6) edge[blue] node[right] {3} (7);
\end{tikzpicture}

\end{minipage}
\hfill 
\begin{minipage}[t][][t]{0.485\textwidth}
\begin{tikzpicture}[scale=0.85,transform shape,>=stealth',shorten
>=1pt,auto,node distance=4cm, thick,main
node/.style={circle,draw,font=\scriptsize,inner sep=1mm}]
 \tikzstyle{every node}=[node distance = 3.5cm,bend angle    = 45,fill          =
gray!30]

  \node[main node] (0) at ({6*45}:3) {(0,0,0)};
  \node[main node] (1) at ({5*45}:3) {(0.0,1)};
  \node[main node] (2) at ({4*45}:3) {(0,1,0)} ;
  \node[main node] (3) at ({3*45}:3) {(0,1,1)};
  \node[main node] (4) at ({2*45}:3) {(1,0,0)};
  \node[main node] (5) at ({1*45}:3) {(1,0,1)} ;
  \node[main node] (6) at ({0*45}:3) {(1,1,0)};
  \node[main node] (7) at ({7*45}:3) {(1,1,1)};
  \node     at (0,-4.5) {$(\Z_2^3, d_{RT})$};

  \path[every node/.style={font=\large,circle}]
  
  (0) edge node[left,pos=0.64] {1} (4)
       edge[red] node [right, pos=0.64] {2} (2)
       edge[red] node [left, pos=0.51] {2} (6)
       edge[blue] node[below] {3} (1)
       edge[blue] node[left=-0.2cm, pos=0.49] {3} (3)
       edge[blue] node[right=-0.2cm] {3} (5)
       edge[blue] node[below] {3} (7)
    (1) edge node [left,pos=0.40] {1} (5)
        edge[red] node [right=-0.14cm, pos=0.57] {2} (3)
        edge[red] node [above, pos=0.4] {2} (7)
        edge[blue] node[left] {3} (2)
        edge[blue] node[left=-0.14cm, pos=0.516] {3} (4)
        edge[blue] node[right=0.1cm,pos=0.42] {3} (6)
        
    (2) edge node [above,pos=0.64] {1} (6)
      edge[red] node [below, pos=0.64] {2} (4)
      edge[blue] node[left] {3} (3)
      edge[blue] node[above=-0.14cm,pos=0.486] {3} (5)
      edge[blue] node[left,below=-0.14cm,pos=0.486] {3} (7)
        
    (3) edge node [left,pos=0.70] {1} (7)
      edge[red] node [below=-0.14cm, pos=0.56] {2} (5)
      edge[blue] node[above] {3} (4)
      edge[blue] node[above=-0.14cm,pos=0.516] {3} (6)
      
    (4) edge[red] node [below, pos=0.5] {2} (6)
       edge[blue] node[above] {3} (5)
       edge[blue] node[right=-0.14cm,pos=0.486] {3} (7)
    
    (5) edge[red] node [left=-0.14cm, pos=0.56] {2} (7)
    edge[blue] node[right] {3} (6)
    
    (6) edge[blue] node[right] {3} (7);
\end{tikzpicture}
\end{minipage}
\end{center}
One can easily check that the map $\varphi$ preserves distances.

Also, note that the associated weight functions $w_{RT} : \Z_2^3 \rightarrow [\![0,3]\!]$ and $w_2 : \Z_8 \rightarrow [\![0,3]\!]$
are given by
\begin{equation*}
    w_{RT}(x) =
    \begin{cases}
      0    & \quad \text{if } x=(0,0,0),\\
      1    & \quad \text{if } x=(1,0,0),\\
      2    & \quad \text{if } x=(a,1,0), \\ 
      3    & \quad \text{if } x=(a,b,1),   
    \end{cases} \qquad \text{and} \qquad 
		 w_2(x) =
    \begin{cases}
      0 & \quad \text{if } x=0,\\
      1 & \quad \text{if } x=4,\\
      2 & \quad \text{if } x=2,6,\\
      3 & \quad \text{if } x=1,3,5,7,
   \end{cases}
\end{equation*}
with $a,b \in \Z_2$. The weight enumerators are thus 
$$\mathcal{W}_{(\Z_2^3,d_{RT})}(t) = \mathcal{W}_{(\Z_8,d_2)}(t) = 4t^3+2t^2+t+1.$$
\end{exam}

\section{Extending isometries of subgroups}
In this section we show that any isometry between metric subgroups can be extended to an isometry between the ambient groups with extended metrics. We recall from Remark \ref{rem inv} that all the metrics considered are $G$-invariant. 
More precisely, we have the following

\begin{thm} \label{Isonotrivial}
Let $G_1$ and $G_2$ be two finite groups with $|G_1|=|G_2|$ and let $H_1 \subsetneq G_1$, $H_2 \subsetneq G_2$ be non-trivial proper subgroups with $|H_1|=|H_2|$. 
Then, any isometry between $H_1$ and $H_2$ can be extended to an isometry between $G_1$ and $G_2$.
\end{thm}

\begin{proof}
Suppose that $(H_1,d_1) \simeq (H_2,d_2)$ and let $\tau:H_1 \rightarrow H_2$ be the isometry. Now, for $i=1,2$, we extend the metrics $d_i$ of $H_i$ to metrics $\tilde d_i$ of $G_i$ as follows:
\begin{equation} \label{d tilde}
\tilde d_i (x,y) :=
    \begin{cases}
      d_i(x,y)   & \qquad \text{if } x-y \in H_i, \\[2mm] 
			\max\limits_{u,v \in H_i} \{ d_i(u,v) \} +1  & \qquad \text{if } x-y \not\in H_i.
    \end{cases}
\end{equation}
Clearly $\tilde d_i(x,y)=0$ if and only if $x=y$ and $\tilde d_i(x,y) = \tilde d_i(y,x)$. We must check that $\tilde d_i$ satisfies the triangular inequality. Let $x,y,z \in G_i$. If $x-y, x-z, z-y \in H_i$ it follows from the triangular inequality from $d_i$. Now, if one of
 $x-z$ or $z-y$ is not in $H_i$, say $x-z$, then 
$$\tilde d_i(x,z) = \max_{u,v \in H_i}\{d_i(u,v)\}+1$$ 
and we have $\tilde d_i(x,z) = \tilde d_i(x,y)$ if $x-y \notin H_i$ or $\tilde d_i(x,z) \ge \tilde d_i(x,y)$ if $x-y \in H_i$, 
and the claim follows.

Now, suppose that $m=|G_1|=|G_2|$ and $h=|H_1|=|H_2|$. 
Let $T_i$ be a complete set of representatives of the right cosets of $H_i$ in $G_i$ for $i=1,2$. 
Consider any bijection $\rho : T_1 \rightarrow T_2$ and define the map 
\begin{align} \label{bij coset}
\begin{aligned}
\qquad G_1 \;\; & \overset{\eta} {\longrightarrow} \;\; G_2 \\
h + g_j \; &\longmapsto \; \tau (h) + \rho(g_j) ,
\end{aligned}
\end{align}
where $g_1, g_2, \ldots, g_{\frac mh}$ are the elements of $T_1$. It is clear that $\eta$ is bijective.

Note that $x, y$ belong to the same coset of $H_1$ if and only if $\eta (x), \eta (y)$ belong to the same coset of $H_2$. 
Therefore we conclude that 
$$ \tilde d_1(x,y) = d_1(x,y) = d_2(\eta (x),\eta(y)) = \tilde d_2(\eta (x),\eta(y))$$
and hence $(G_1, \tilde d_1) \simeq (G_2, \tilde d_2)$, as it was to be shown.
\end{proof}

\begin{rem}
In the previous proof, the isometry $\eta$ 
given by \eqref{bij coset} is not unique, since $\eta=\eta_{\rho}$ depends on the bijection $\rho$ between the complete set of representatives of right cosets $T_1$ on $G_1$ and $T_2$ on $G_2$ chosen. 
However, two such metrics differ by a distance preserving map. That is, if $\rho$ and $\rho'$ are two bijections from $T_1$ to $T_2$ then there is some $f\in \Gamma(G,\tilde d_2)$ such that $\eta_{\rho} = f \circ \eta_{\rho'}$. In fact, 
if $f= \eta_\rho \circ \eta_{\rho'}^{-1}$ then  
$$\tilde d_2(f(x),f(y)) = \tilde d_2 \big( \eta_\rho(\eta_{\rho'}^{-1}(x)), \eta_\rho(\eta_{\rho'}^{-1}(y)) \big) =
\tilde d_1(\eta_{\rho'}^{-1}(x), \eta_{\rho'}^{-1}(y)) = \tilde d_2(x,y)$$
and hence $f$ preserves distances.  
\end{rem}

A direct consequence of this result is that every pair of groups of the same size are isometric.
For groups of prime cardinality, the isometry is trivial in the sense that both metrics are Hamming metrics.

\begin{coro} \label{ext metric}
Let $G_1$ and $G_2$ be groups of the same cardinality.  
Then, there exists an isometry $\phi : (G_1, d_1) \rightarrow (G_2, d_2)$ where 
$d_1$ and $d_2$ are certain metrics in $G_1$ and $G_2$. Moreover, if $|G_i|$ is not prime then $d_i$ can be chosen in such a way that $d_i \ne d_{_{\Ham}}$ for $i=1,2$.
\end{coro}

\begin{proof}
Suppose $m=|G_1|=|G_2|$. If $m$ is not prime, consider a prime $p$ dividing $m$. Then, there are non-trivial proper subgroups $H_1<G_1$ and 
$H_2<G_1$ with $p=|H_1|=|H_2|$. 
Considering the Hamming metric in both 
$H_1$ and $H_2$ it is clear that these subgroups are isometric, that is $(H_1,d_{\Ham}) \simeq (H_2,d_{\Ham})$. 
From Theorem \ref{Isonotrivial}, this isometry lifts to an isometry 
$$(G_1, d_1) \simeq (G_2, d_2),$$ 
where the metric $d_i = \tilde d_{\Ham}$ for $i=1,2$ and $\tilde d$ is as in \eqref{d tilde}.  
That is, $d_i(x,x)=0$, 
\begin{equation} \label{dGH}
d_i(x,x+h)=1 \quad \text{if $h\in H\smallsetminus \{0\}$} \qquad \text{and} \qquad
d_i(x,x+g)=2  \quad \text{if $g\in G\smallsetminus H$,}
\end{equation}
for any $x\in G_i$ and $i=1,2$. 

If $m=p$, then $G_1\simeq G_2 \simeq \Z_p$ and they are trivially isometric with the Hamming metrics. This completes the proof.
\end{proof}

\begin{rem}
By Corollary \ref{ext metric}, for any $m,n \geq 2$ there exist, for instance, non-trivial isometries  
$\Z_{m^n} \simeq (\Z_m)^n$, $\F_{q^n} \simeq (\F_q)^n$, $\D_{m} \simeq \Z_{2m}$, etc.
\end{rem}

Let $H \subsetneq G$ be a non-trivial proper subgroup and $d$ a metric in $H$. In the sequel we will denote by
\begin{equation} \label{dext} 
\tilde d= Ext_H^G(d)
\end{equation}
(or simply $Ext_H(d)$ when $G$ is understood) 
the metric in $G$ induced by the extension given in Theorem \ref{Isonotrivial} (see \eqref{d tilde}). We will call this the \textit{extended metri}c of $d$ from $H$ to $G$.

We now illustrate the previous theorem for groups of small order $n=4, 6, 8$.
\begin{exam}[$n=4$]
Let $G_1= \Z_4$ and $G_2= \Z_2^2$. It is known that $(\Z_4, d_{_{Lee}})$ is isometric to $(\Z_2^2, d_{\Ham}^2)$ 
via the Gray map. We now show that they are isometric by using 
Theorem~\ref{Isonotrivial}. 

Consider the subgroups 
$H_1=\Z_2 \subsetneq \Z_4$ and $H_2 = \Z_2 \times \{0\} \subsetneq \Z_2 \times \Z_2$, both with the Hamming metric $d_{\Ham}$. These subgroups are isometric via the inclusion map 
$\iota: \Z_2 \rightarrow \Z_2 \times \{0\}$ given by $x \mapsto (x,0)$. 
By Theorem \ref{Isonotrivial} and \eqref{dext} we have that 
$$\big( \Z_4, \tilde d_1 = Ext_{\Z_2}^{\Z_4}(d_{\Ham}) \big) \simeq 
		\big( \Z_2^2, \tilde d_2 = Ext_{\Z_2 \times \{0\} }^{\Z_2 \times \Z_2}(d_{\Ham}) \big).$$ 

Notice that $Ext_{\Z_2}^{\Z_4}(d_{\Ham}) = d_2$, the $2$-adic metric, and 
$Ext_{\Z_2 \times \{0\} }^{\Z_2 \times \Z_2}(d_{\Ham}) = d_{RT}$, the $RT$-metric.  
In fact, by \eqref{qadica} and \eqref{RT metric} we have 
$$d_2(x,y) = \min\limits_{0\le i \le 2} \{i : 2^{2-i} \mid x-y\} \qquad \text{and} \qquad 
d_{RT}(x,y) = \max\limits_{1\le i \le 2} \{i : x_i-y_i \ne 0\}$$
respectively. Hence, by \eqref{d tilde} or \eqref{dGH}, 
 for any $x, y \in \Z_4$ we have 
\begin{align*}
\tilde d_1 (x,y) = 1 = d_2(x,y)  & \qquad \text{if } x-y=2, \\
\tilde d_1 (x,y) = 2 = d_2(x,y)  & \qquad \text{if } x-y=1,3, 
\end{align*}
while for any $u, v \in \Z_2^2$ we get
 \begin{align*}
\tilde d_2 (u,v) = 1 = d_{RT}(u,v)  & \qquad \text{if } u-v=(1,0), \\
\tilde d_2 (u,v) = 2 = d_{RT}(u,v)  & \qquad \text{if } u-v=(0,1), (1,1).
\end{align*}

We wish to point out that, although $d_2$ and $\tilde d_{RT}$ are different from $d_{_{Lee}}$ and $d_{\Ham}$, these metrics are correspondingly equivalent, i.e.\@ $d_2 \simeq d_{_{Lee}}$ and $d_{RT} \simeq d_{\Ham}$, in a precise sense that we will not discuss here (this will be treated in another work). 
\end{exam}

\begin{exam}[$n=6$]
We now consider the groups $\Z_6$ and $\D_3$. By Theorem \ref{Isonotrivial}, we have two different isometries between them, taking as $H_1, H_2$ subgroups of order $2$ or $3$, respectively. 
Namely,
$$(\Z_6, Ext_{\Z_2}(d)) \simeq (\Sym_3, Ext_{\langle \rho \rangle}(d)) \qquad \text{and} \qquad  
(\Z_6, Ext_{\Z_3}(d)) \simeq (\Sym_3, Ext_{\langle \tau \rangle}(d)),$$
where $\rho$ is a $2$-cycle, $\tau$ a $3$-cycle and $d$ is the Hamming metric in $\Z_2$ and $\Z_3$, respectively.

Apart from the Hamming and Lee metrics, in addition we have the metrics obtained by the previous subgroup construction.
The corresponding weight functions and enumerators are given by

\renewcommand{\baselinestretch}{1.6} 
$$\begin{tabular}{|c|cccccc|c|}
\hline
$\Z_6$     & $0$ & $1$ & $2$ & $3$ & $4$ & $5$    & enumerator \\  \hline 
$w_{\Ham}$  & $0$ & $1$ & $1$ & $1$ & $1$ & $1$ & $5t+1$ \\
$w_{\Z_2}$ & $0$ & $2$ & $2$ & $1$ & $2$ & $2$    & $4t^2+t+1$ \\
$w_{\Z_3}$ & $0$ & $2$ & $1$ & $2$ & $1$ & $2$    & $3t^2+2t+1$ \\
$w_{Lee}$  & $0$ & $1$ & $2$ & $3$ & $2$ & $1$    & $t^3+2t^2+2t+1$ \\ \hline 
\end{tabular}$$
and  
$$\begin{tabular}{|c|cccccc|c|} \hline 
$\Sym_3$   & $id$ & $(12)$ & $(13)$ & $(23)$ & $(123)$ & $(132)$ & enumerator \\ \hline
$w_{\Ham}$  & $0$ & $1$ & $1$ & $1$ & $1$ & $1$ & $5t+1$ \\
$w_{\langle (12) \rangle}$ & $0$ & $1$ & $2$ & $2$ & $2$ & $2$ & $4t^2+t+1$ \\
$w_{\langle \tau \rangle}$ & $0$ & $2$ & $2$ & $2$ & $1$ & $1$ & $3t^2+2t+1$ \\ \hline
\end{tabular}
$$
where $\tau$ is any $3$-cycle.
\end{exam}

\begin{exam}[$n=8$]
By Corollary \ref{ext metric}, all the groups of the same size are isometric to each other. Thus, for instance, we have 
$$\Z_2^3 \simeq \Z_2 \times \Z_4 \simeq \Z_8 \simeq \D_4 \simeq \Q_8.$$ 
In fact, note that all these groups have at least one isomorphic copy of $\Z_2$ as a subgroup. 
Thus, if we take any pair $G_1, G_2 \in \{\Z_2^3, \Z_2 \times \Z_4, \Z_8, \D_4,\Q_8\}$, with the trivial identifications, we then have 
$$ (G_1, Ext_{\Z_2}^{G_1}(d_{\Ham})) \simeq (G_2, Ext_{\Z_2}^{G_2}(d_{\Ham})).$$
The associated weights $w_{\Z_2}$ are given by 
\renewcommand{\arraystretch}{1.2}
$$\begin{tabular}{|c|cccccccc|}
\hline 
$\Z_8$     & $0$ & $1$ & $2$ & $3$ & $4$ & $5$ & $6$ & $7$ 
\\  \hline 
$w$        & $0$ & $2$ & $2$ & $2$ & $1$ & $2$ & $2$ & $2$ 
\\   \hline 
\hline
$\Z_2^{3}$     & $(0,0,0)$ & $(1,0,0)$ & $(0,1,0)$ & $(0,0,1)$ & $(1,1,0)$ & $(0,1,1)$ & $(1,0,1)$ & $(1,1,1)$  
\\  \hline 
$w$        & $0$ & $1$ & $2$ & $2$ & $2$ & $2$ & $2$ & $2$ 
\\  \hline 
\hline
$\Z_2\times\Z_4$     & $(0,0)$ & $(1,0)$ & $(1,1)$ & $(1,2)$ & $(1,3)$  & $(0,1)$ & $(0,2)$ & $(0,3)$  
\\  \hline 
$w$        & $0$ & $1$ & $2$ & $2$ & $2$ & $2$ & $2$ & $2$ 
\\  \hline 
\hline
$\D_4$     & $e$ & $\rho$ & $\rho^{2}$ & $\rho^{3}$ & $\tau$ & $\rho\tau$ & $\rho^{2}\tau$ & $\rho^{2}\tau$ 
\\  \hline 
$w$        & $0$ & $2$ & $2$ & $2$ & $1$ & $2$ & $2$ & $2$ 
\\  \hline 
\hline
$\Q_8$     & $1$ & $-1$ & $i$ & $-i$ & $j$ & $-j$ & $k$ & $-k$  
\\  \hline 
$w$        & $0$ & $1$ & $2$ & $2$ & $2$ & $2$ & $2$ & $2$  
\\  \hline 
\end{tabular}$$
The weight enumerator is $\mathcal{W}_{(G,w_{\Z_2})}(t) = 6t^2+t+1$, where $G$ is any group of order 8. 

We now compute the weight enumerators for all the subgroups of all the groups of order 8. Since isomorphic subgroups give the same metric, we consider subgroups up to isomorphism.
It is clear that $\Z_4$ is a subgroup of $G_1 \in \{\Z_8, \Z_2 \times \Z_4, \D_4, \Q_8\}$, that $\Z_2 \times \Z_2$ is a subgroup of $G_2 \in \{\Z_2^3, \Z_2 \times \Z_4, \D_4\}$ and that $\Z_2$ is a subgroup of $G_3$, any group of order 8. 

Thus, the weight enumerators for the corresponding extended metrics are as follows
\begin{equation*}
\begin{aligned}
\mathcal{W}_{(G_1,Ext_{\Z_4}(d_{_{Ham}}))} (t)     & = 4t^2+3t+1, \\  
\mathcal{W}_{(G_1,Ext_{\Z_4}(d_{_{Lee}}))} (t)     & = 4t^3+t^2+2t+1, \\  
\mathcal{W}_{(G_2,Ext_{\Z_2^2}(d_{_{Ham}}))} (t)   & = 4t^2+3t+1, \\  
\mathcal{W}_{(G_2,Ext_{\Z_2^2}(d_{_{Ham}}^2))} (t) & = 4t^3+t^2+2t+1, \\  
\mathcal{W}_{(G_3,Ext_{\Z_2}(d_{_{Ham}}))} (t)     & = 6t^2+t+1.   
\end{aligned}
\end{equation*}

\end{exam}

\section{Chain metrics and chain isometries}
In this section we will consider chain metrics and chain isometries on groups with chains of subgroups, generalizing the construction and results of the previous section.

\begin{defi} \label{defi chm}
Let $G$ be a group and $\C$ a chain of subgroups of $G$, 
\begin{equation} \label{chain} 
\langle 0 \rangle = H_0 \subsetneq H_1 \subsetneq \cdots\subsetneq H_n = G.
\end{equation}
The \textit{chain metric} on $G$ associated to $\mathcal{C}$ is defined by 
\begin{equation} \label{chain metric} 
d_\C(x,y) = i \qquad \text{if } x-y \in H_i \smallsetminus H_{i-1}
\end{equation}
for $i=0,\ldots, n$. Here 
we use the convention $H_{-1}=\varnothing$.
\end{defi}

We now check that $d_\C$ is indeed a metric. We only have to show that the triangular inequality holds. 
Let $x,y,z\in G$ and suppose that $d(x,y)=i$, $d(x,z)=j$ and $d(z,y)=k$. Thus,
$x-y \in H_{i} \smallsetminus H_{i-1}$, $x-z\in H_{j} \smallsetminus H_{j-1}$ and $z-y \in H_{k} \smallsetminus H_{k-1}$.
We can assume that $k\geq j$, therefore 
$$ x-y = (x-z)-(y-z) \in H_{k} \smallsetminus H_{k-1}.$$
This implies that $d(x,y) \le k \le d(x,z) + d(z,y)$, as we wanted to see. 

As a  direct consequence of Definition \ref{defi chm}, the weight enumerator of $G$ with the chain metric is given by 
\begin{equation} \label{enum chain}
\mathcal{W}_{(G, d_\C)} (x) = \sum_{i=0}^n (|H_i|-|H_{i-1}|) \, x^i .
\end{equation}

\begin{rem}
It is worth noting that the $q$-adic metric in $\Z_{q^n}$ and the $RT$-metric in $\Z_q^n$ are chain metrics.

\noindent ($i$) 
Let $G=\Z_{q^n}$ and consider the chain of subgroups $\C$ given by
$$ \langle 0 \rangle \subsetneq \Z_{q} \subsetneq \Z_{q^2} \subsetneq \cdots \subsetneq \Z_{q^{n}},$$ 
where we are identifying $\Z_{q^i}$ with $\inner{q^{n-i}} = q^{n-i} \Z_{q^n}$ for $i=0,\ldots,n$. 
In fact, since for $x,y \in G$ we have  
$$ x-y \in \langle q^{n-i} \rangle \smallsetminus \langle q^{n-(i-1)} \rangle \quad \Leftrightarrow \quad 
q^{n-i} \mid x-y \:\: \text{ and } \:\: q^{n-(i-1)} \nmid x-y$$
then 
$$d_\C(x,y) = \min\limits_{0\le i \le n} \{i: q^{n-i} \mid x-y \} = d_q(x,y)$$ 
holds for any $x,y\in G$, by \eqref{qadica} and \eqref{chain metric}.

\noindent ($ii$) 
Let $G=\Z_q^n$ and consider the following chain of subgroups $\C$,
$$ \langle 0 \rangle \subsetneq \Z_q  \subsetneq \Z_q^2  \subsetneq \cdots \subsetneq \Z_q^n,$$ 
where by abuse of notation $\Z_q^i$ denotes $\Z_q^i \times \{0\}^{n-i}$ for $i=1,\ldots,n$.
In fact, since for $x,y \in G$ we have  
\begin{eqnarray*}
x-y \in \Z_q^i \smallsetminus \Z_q^{i-1} & \quad \Leftrightarrow \quad &
x_i-y_i \ne 0 \quad \text{ and } \quad x_i-y_i \in \Z_q^{i} \\[1mm]
& \quad \Leftrightarrow \quad &  x_i-y_i \ne 0 \quad \text{ and } \quad  x_j-y_j =0 \quad \text{ for } j > i.
\end{eqnarray*}
Then, 
$$d_\C(x,y) = \max\limits_{1\le i \le n} \{i: x_i-y_i \ne 0 \} = d_{RT}(x,y)$$ 
holds for any $x,y\in G$, by \eqref{RT metric} 
and \eqref{chain metric}.

Notice that the weight enumerators given in \eqref{enum qadic} are of the form \eqref{enum chain}.
\end{rem}

We now exhibit another chain metric. 
Let $G$ be a finite group and $r,n$ positive integers. Consider the following chain of groups 
\begin{equation} \label{cadena Gr's} 
\C \,: \quad G \subset G^r \subset G^{r^2} \subset G^{r^3} \subset  \cdots \subset G^{r^n},
\end{equation}
where the inclusions are given by the diagonal maps $\delta_i$. For instance, $\delta_0 :G \rightarrow G^r$ is given by 
$x \mapsto (x,x,\ldots,x)$ with $x$ repeated $r$-times, 
$\delta_1 :G^r \rightarrow G^{r^2}$ is given by 
$$(x,x,\ldots,x) \mapsto ((x,x,\ldots,x),(x,x,\ldots,x),\ldots,(x,x,\ldots,x)),$$ 
and so on.
The chain metric $d_\C$ associated to $\C$ is given as follows.
If $x=(x_1,...,x_{r^n}) \in G^{r^n}$ the weight function associated to $\C$ is given by 
\begin{equation} \label{peso diag}
w_\C (x) = \min_{0\le i \le n} \{ i : x_j=x_j \: \text{ if } \: j\equiv k \!\!\! \pmod{r^{i-1}} \}.
\end{equation}
Let $m=r^n$. The group $\Sym_m$ acts on $G^m$ by permutation of coordinates. If $\sigma =(12\cdots m) \in \Sym_m$, one can check that this is equivalent to  
\begin{equation} \label{peso diag2}
w_\C (x) = \min_{1 \le i \le n} \{ i : \sigma^{r^{i-1}}(x) = x\}
\end{equation}
for $x\ne 0$ and $w_\C(x)=0$ if $x=(0,0,\ldots,0)$.
The chain metric is given by $d_\C (x,y) = w_\C(x-y)$.
We call this the\textit{ diagonal chain metric} of $G$, and we denote it by $d_{\Delta}$.

\begin{exam} \label{diagonal}
Take $G=\Z_2$, $r=2$ and $n=3$ in \eqref{cadena Gr's}. Then we have  
$$ \langle 0 \rangle \subset \Z_2 \subset \Z_2^2 \subset \Z_2^4 \subset \Z_2^8.$$ 
The possible weights in $\Z_2^8$ are $0,1,2,3,4$ given by
$$w_\C(x) = \begin{cases} 
0 & \qquad \text{if } x=(0,0,0,0,0,0,0,0), \\[1mm] 
1 & \qquad \text{if } x=(1,1,1,1,1,1,1,1), \\[1mm] 
2 & \qquad \text{if } x=(x_1,x_2,x_1,x_2,x_1,x_2,x_1,x_2) \text{ with } x_1 \ne x_2, \\[1mm] 
3 & \qquad \text{if } x=(x_1,x_2,x_3,x_4,x_1,x_2,x_3,x_4) \text{ with } x_1 \ne x_3 \text{ or } x_2 \ne x_4,\\[1mm] 
4 & \qquad \text{otherwise}.
\end{cases}$$ 
This is in coincidence with expressions \eqref{peso diag} and \eqref{peso diag2}.
It is clear that the corresponding weight enumerator is given by
$$\mathcal{W}_{(\Z_2^8, d_{\Delta})}(t) = 240 t^4 + 12 t^3 + 2 t^2 + t +1.$$
Compare with the weight enumerator
$$\mathcal{W}_{(\Z_2^8, d_{RT})}(x) = 128 t^8 + 64 t^7 + 32 t^6 + 16 t^5 + 8 t ^4 + 4t^3 + 2t^2 +t+1$$
of $\Z_2^8$ with the $RT$-metric.
\hfill $\lozenge$
\end{exam}

Let $\mathcal{C}$ denote a chain of subgroups as in \eqref{chain}
and let $d$ be a metric in $H_1$. The metric in $G$ obtained by repeated extensions is 
\begin{equation} \label{extension}
\widetilde{d} = Ext_\mathcal{C}(d) = Ext_{H_{n-1}}^{H_n} \circ \cdots \circ Ext_{H_1}^{H_2}(d).
\end{equation}

\begin{rem}
In the above situation, if in \eqref{extension} we take the Hamming metric in $H_1$, the extended metric turns out to be the chain metric of $\C$, i.e.\@ 
$$\tilde d_{\Ham}= d_\C.$$
\end{rem}

\subsection*{Chain isometries}
We now consider isometries between whole chains of groups. 
Let $\mathcal{C}$ and $\mathcal{C}'$ be two chains of subgroups of the same length of groups $G$ and $G$' respectively, say 
$H_1 \subsetneq H_2 \subsetneq \cdots \subsetneq H_n=G$ and  
$H_1' \subsetneq H_2' \subsetneq \cdots  \subsetneq H_n'=G'$.

\begin{defi}
We say that $\mathcal{C}$ is \textit{isometric} to $\mathcal{C}'$, denoted $\C \simeq \C' $, if for every $i=1,\ldots,n$ there are metrics $d_i$ of $H_i$ and $d_i'$ of $H_i'$ such that $(H_i,d_i) \simeq (H_i',d_i')$. 
The groups $G$ and $G'$ are said to be \textit{chain isometric} if they admit isometric chains.
\end{defi}

That is, if two chains $\mathcal{C}$ and $\mathcal{C}'$ are isometric we have 
\begin{equation} \label{diagramchain}
\begin{matrix}
H_1  & \subsetneq & H_2 & \subsetneq & \cdots & \subsetneq & H_n = G  \\[.75mm]
\mid \simeq & &  \mid \simeq & & & & \mid \simeq \\[.75mm]
H_1' & \subsetneq & H_2' & \subsetneq & \cdots & \subsetneq & H_n' = G'  
\end{matrix}
\end{equation}

We now show that chains of the same length and corresponding sizes are isometric.
\begin{lema} 
	Let $G$ and $G'$ be groups with chains of subgroups $\C$ and $\C'$, respectively given by
	$\langle 0 \rangle \ne H=H_1 \subsetneq H_2 \subsetneq \cdots \subsetneq H_n= G$ and $\langle 0 \rangle \ne H'=H_1' \subsetneq H_2' \subsetneq \cdots \subsetneq H_n'= G'$. If $|H_i|=|H_i'|$ for $1\le i \le n$ then we have the chain isometry
	$$(G,d_{\C}) \simeq (G', d_{\C'}).$$ 
\end{lema}

\begin{proof}
	Since $|H_1|=|H_1'|$ there is a bijection $\eta : H_1 \rightarrow H_1'$ inducing the trivial isometry 
	$(H_1,d_{\Ham}) \simeq (H_1', d_{\Ham})$. 
	By applying part ($b$) of Theorem  \ref{Isonotrivial} we can lift this isometry to get $(H_2, Ext_{H_1}^{H_2}(d_{\Ham})) \simeq 
	(H_2', Ext_{H_1'}^{H_2'}(d_{\Ham}))$. Repeating this lifting procedure we obtain that $\mathcal{C}$ and $\mathcal{C}'$ are isometric chains with the extended metrics. 
\end{proof}

\begin{exam} \label{ejemplin}
	(\textit{i}) The isometry $\Z_{q^n} \simeq (\Z_q)^n$ given explicitly in Section \ref{Zqn}, 
	can be seen as a chain isometry. In fact,  the chains
$$	\begin{matrix}
	\Z_{q} &\subset & \Z_{q^2}& \subset &\cdots& \subset &\Z_{q^{n-1}}& \subset &\Z_{q^n} \\[2mm]
	\mid   &     & \mid    &   &   &  &\mid &   & \mid \\[2mm]	
    \Z_q   & \subset & \Z_q^2  & \subset & \cdots & \subset &\Z_q^{n-1} &\subset &\Z_q^n 
	\end{matrix}$$
	are isometric by the previous lemma.

	\noindent
	(\textit{ii}) There is a chain isometry $\Z_{q^n} \simeq \ff_q^n$ given by the chains 
	$\Z_q  \subset \Z_{q^2} \subset \cdots \subset \Z_{q^n}$ and $\ff_q \subset \ff_q^2  \subset \cdots \subset \ff_q^{n}$.
	In fact, any bijection between $\ff_q$ and $\Z_q$ with the Hamming metrics induces a chain isometry between $\ff_q^n$ and $\Z_{q^n}$. 
	One can replace $\ff_q$ and $\Z_q$ by any group $C_q$ of order $q$.
\end{exam}

\begin{exam} [\textit{Galois fields and rings}]
Let $p$ be a prime and $r_1, r_2, \ldots, r_n$ be positive integers such that $r_1 \mid r_2 \mid \cdots \mid r_n$. 
Consider the Galois rings $R_i=GR(p^{k},{r_i})$ for $i=1,\ldots, n$.
Then we have the isometric chains of rings
\begin{equation*} 
\begin{matrix}
GR(p^{k},{r_1})  & \subset & GR(p^{k},{r_1})^{\frac{r_2}{r_1}} & \subset & \cdots & \subset & GR(p^{k},{r_1})^{\frac{r_n}{r_1}}  \\[1mm]
\mid  & & \mid  &&&& \mid  \\[1mm]
GR(p^{k},{r_1}) & \subset & GR(p^{k},{r_2}) & \subset & \cdots & \subset & GR(p^{k},{r_n}) 
\end{matrix}
\end{equation*}
and, in particular taking $k=1$, $GR(p,{r_i})\simeq\F_{p^{r_i}}$, so this becomes
\begin{equation*} \label{fields}
\begin{matrix}
\F_{p^{r_1}}  & \subset & (\F_{p^{r_1}})^{\frac{r_2}{r_1}} & \subset & \cdots & \subset &(\F_{p^{r_1}})^{\frac{r_n}{r_1}} \\[1mm]
\mid  & & \mid  &&&& \mid  \\[1mm]
\F_{p^{r_1}} & \subset & \F_{p^{r_2}} & \subset & \cdots & \subset & \F_{p^{r_n}}
\end{matrix}
\end{equation*}
\end{exam}

\subsubsection*{Geometric chains}
Let $\mathcal{C}$ be a chain $\langle 0 \rangle= H_0 \subset H_1 \subset H_2 \subset \cdots \subset H_n=G$ with the sizes of the terms in geometric progression, that is  
\begin{equation} \label{cocientes} 
[H_{i} : H_{i-1}]=m \qquad \text{for } i=1,\ldots,n. 
\end{equation}
We will call this a \textit{geometric chain}.

\begin{prop} \label{prop geom}
If $G$ admits a geometric chain $\C$ of subgroups
$H=H_1 \subsetneq 
\cdots \subsetneq H_n= G$ with $H\ne \langle 0 \rangle$ then we have the isometry
$$(G,d_{\C}) \simeq (H^n, d_{RT}).$$ 
\end{prop}

\begin{proof}
Let $\C$ be the chain $H=H_1 \subsetneq H_2 \subsetneq \cdots \subsetneq H_n= G$ and consider the geometric chain $\mathcal{C}'$ given by $H \subset H^2 \subset H^3 \subset \cdots \subset H^n$. Starting from the trivial isometry 
$id : H \rightarrow H$ with the Hamming metrics and applying Theorem  \ref{Isonotrivial}, 
we get that $\mathcal{C}$ and $\mathcal{C}'$ are isometric chains. In particular, $G\simeq H^n$ and 
$$(\widetilde{d}_{\Ham})^n = d_{\mathcal{C}'} = d_{RT}$$
as we wanted to see. 
\end{proof}

\begin{rem}
The isometries $(\Z_{q^n}, d_q) \simeq (\Z_q^n, d_{RT})$ and $(\Z_{q^n}, d_q) \simeq (\ff_q^n, d_{RT})$ given in Example \ref{ejemplin} are instances of geometric chains and of chain isometries given by the previous proposition. 
\end{rem}

We now show that the result in Theorem \ref{IsoRT}, i.e.\@ that $(\Z_{q^n},d_q) \simeq (\Z_q^n, d_{RT})$, can be generalized to any pair of groups $G$ and $H^n$ of order $q^n$, 
with $G$ and $H$ not necessarily cyclic. 

\begin{thm}
Let $q$ be a prime power and $G, H$ groups with $|G|=q^n$ and $|H|=q$. Then, 
\begin{equation} \label{G iso Hn} 
(G, d_\mathcal{C}) \simeq (H^n, d_{RT}),
\end{equation}
where $d_\mathcal{C}$ is the chain metric associated to some geometric chain of length $n$. 
\end{thm}

\begin{proof}
Since $|G|=q^n$, by Sylow's theorems we get that $G$ has a geometric chain $\mathcal{C}$ of length $n$, say 
$\langle 0 \rangle \subset H_1 \subset \cdots \subset H_n= G$. By Proposition \ref{prop geom} we have that
$$(G, d_\mathcal{C}) \simeq (H_1^n, d_{RT}).$$
On the other hand, since $|H_1|=|H|$ there is a bijection $\tau : H_1 \rightarrow H$ which extends to $\tau : H_1^n \rightarrow H^n$ and induces the isometry 
$$(H_1^n,d_{RT}) \simeq (H^n,d_{RT}).$$ 
In fact,
$$d_{RT}(\tau(x), \tau(y)) = \max_{1\le i \le n} \{ i: \tau(x_i) \ne \tau(y_i) \} = 
\max_{1\le i \le n} \{ i: x_i \ne y_i \} = d_{RT}(x, y).$$
This implies the result.
\end{proof}

The theorem implies, for instance, that there exists a metric $d$ in the generalized quaternion group $\Q_{2^{n}}$ of order $2^n$ and a metric $d'$ in the dihedral group $\D_{2^{n-1}}$ of order $2^{n}$ such that
$$ (\Z_{2^{n}},d_2) \simeq  (\Q_{2^{n}},d) \simeq (\D_{2^{n-1}},d') \simeq (\Z_2^{n},d_{RT}).$$
Also, in the above list one can add all the groups $\Z_{2^i} \times \Z_{2^{n-i}}$ with some metrics $d_{(i)}$ for $i=1,\ldots, n-1$.

\section{Block Rosenbloom--Tsfasman metric}
We will next extend the result for geometric chains given in the previous section for groups with arbitrary chains. 
For this, we must first consider a generalization of the $RT$-metric. 

\begin{defi} \label{brt}
Let $X$ be a group and $n\in \N$. Given a partition $n= m_1 + \cdots + m_r$ consider $X^n = X^{m_1} \times \cdots \times  
X^{m_r}$. We write $x=(\tilde x_1,\ldots,\tilde x_r)$ for an element in $X^n$, where $\tilde x_i \in X^{m_i}$ for any $i$.
We define the \textit{block Rosenbloom--Tsfasman} metric (or $BRT$-\textit{metric}) on $X^n$ as 
\begin{equation} \label{BRT metric}
d_{BRT} (x,y) = \max_{1\le i \le r} \{i : \tilde x_i \ne \tilde y_i \}.
\end{equation}  
\end{defi}
Note that for $r=n$, then $m_1=\cdots =m_n=1$ and hence the $BRT$-metric is just the $RT$-metric. 
Also, notice that this metric can be seen as the \textsl{block poset metric} 
(see \cite{BlockPoset}) associated to the chain poset $1 \preceq 2 \preceq \cdots \preceq  r$.

\begin{thm} \label{prop no geo}
Let $H$ be a proper subgroup of a group $G$ and $\mathcal{C}$ a chain of subgroups with initial term $H$. 
Then we have 
\begin{equation} \label{index chain}
(G,d_\mathcal{C}) \simeq (H^{[G:H]}, d_{BRT}).
\end{equation} 
\end{thm}

\begin{proof}
Suppose $\mathcal{C}$ is the chain $H_1 = H \subset H_2 \subset \cdots \subset H_n =G$.
Consider the group $G'= H^{[G:H]}$. 
We will construct a chain $\mathcal{C'}$ in $G'$ of length $n$, 
say $H_1'=H \subset H_2' \subset \cdots \subset H_n' =G'$, such that 
$|H_i'| = |H_i|$ for all $i=1,\ldots,n$.
Consider $H_2'=H^{[H_2:H_1]}$, 
$$H_3'= (H_2')^{[H_2:H_1]} = (H^{[H_2:H_1]})^{[H_3:H_2]} = H^{[H_3:H_1]}$$ 
and in general for every $1\le i \le n$ take
$$H_i' = H^{[H_i:H_1]}.$$ 
It is clear that $|H_i'|=|H_i|$ for $i=1,\ldots, n$. 

By Theorem \ref{Isonotrivial}, the trivial isometry $\varphi_1=id: (H_1,d_{\Ham}) \rightarrow (H_1',d_{\Ham})$ can be lifted to an isometry 
$$\varphi_2: (H_2, Ext_{H_1}^{H_2}(d_{\Ham})) \rightarrow (H_2', Ext_{H_1'}^{H_2'}(d_{\Ham})).$$ 
By iterating this process we arrive at an isometry 
$$\varphi_n : (H_n, Ext_{H_{n-1}}^{H_n} \circ \cdots \circ Ext_{H_1}^{H_2}(d_{\Ham})) \rightarrow (H_n', Ext_{H_{n-1}'}^{H_n'} \circ \cdots \circ  Ext_{H_1'}^{H_2'}(d_{\Ham})).$$ 
That is, we have 
$$(G, d_\mathcal{C}) \simeq (H^{[G:H]}, d_{\mathcal{C}'}).$$ 

It only remains to show that the chain metric $d_{\mathcal{C}'}$ is the $BRT$-metric.
Put $r=[G:H]$ and $r_i=[H_i:H_{i-1}]$ for $i=1,\ldots,n$ (where $H_{-1}=\langle 0 \rangle$). 
Consider the natural decomposition
$H^r = H^{r_1}  
\times \cdots \times H^{r_n}$. 
If $x\in H^r$ then $x=(\tilde x_1, \ldots, \tilde x_n)$ with $\tilde x_i \in H^{r_i}$ for any $i$.
Since $$d_{BRT}(x,y) = \max_{1\le i \le n} \{ i : \tilde x_i \ne \tilde y_i\}$$ 
one can check that for $i=1, \ldots,n$ we have
$$d_{BRT}(x,y) =i \qquad \Leftrightarrow \qquad d_{\mathcal{C}'}(x,y) =i .$$
Hence the metrics coincide and the result thus follows.
\end{proof}

\begin{exam}
Let $G=\Z_{q^n}$ and $H=\F_q$, $n\geq 2$. By the previous theorem, if we take the chains
$$\C: \quad \langle 0 \rangle \subset \Z_q \subset \Z_{q^n} \qquad \text{ and } \qquad \C' :\quad  \langle 0 \rangle \subset \F_q \subset \F_q^n$$ 
and we consider the decomposition $\F_q^n = \F_q \times \F_q^{n-1}$ we get 
$$(\Z_{q^n}, d_\C) \simeq (\F_q^n, d_{BRT}).$$
Note that the weight function associated to $\C$ is 
$$w_\C = \begin{cases}
0 & \qquad \text{ if } x=0, \\
1 & \qquad \text{ if } x \in q^{n-1}\Z_{q^n}\smallsetminus \{0\}, \\
2 & \qquad \text{ if } x \in \Z_{q^n} \smallsetminus q^{n-1}\Z_{q^n}. 
\end{cases}$$
Now, by properly rescaling this weight, we get the following, 
$$\tilde w_{\C}  = \begin{cases}
0 & \qquad \text{ if } x=0, \\
q^{n-1} & \qquad \text{ if } x \in q^{n-1}\Z_{q^n}\smallsetminus \{0\}, \\
q^{n-2}(q-1) & \qquad \text{ if } x \in \Z_{q^n} \smallsetminus q^{n-1}\Z_{q^n}.
\end{cases}$$
Thus, we obtain
\begin{equation} \label{iso hom brt}
(\Z_{q^n}, \tilde d_{\C}) \simeq (\F_q^n, \tilde d_{BRT}),
\end{equation}
where $\tilde d_{BRT}$ is a rescaled metric obtained from $d_{BRT}$. It is easy to check that 
the rescaled metrics are also metrics. 
In the case when $q=p$ is prime, the metric $\tilde d_{\C}$ coincides with the \textsl{homogeneous metric} (see \cite{Hom}) defined over the ring $\Z_p^{n}$,
\begin{equation}
(\Z_{p^n}, d_{\scriptscriptstyle Hom}) \simeq (\F_p^n, \tilde d_{BRT}) .
\end{equation}
\end{exam}

\begin{rem}
As in the previous example, we have the isometry $(\Z_{q^n}, d_{Hom}) \simeq (\F_q^n, \tilde d_{BRT})$ for $q=p^{r}$.
Consider the $q$-ary first order Reed-Muller code $RM(1,q^{n-1})$ and let $G$ be any generating matrix of the code whose first row is the all ones vector $(1,1,\ldots,1)$.  The code $RM(1,q^{n-1})$ lies in $\F_q^{q^{n-1}}$ with the Hamming metric, and right multiplication by $G$
encodes the space $\F_q^{n}$ into $RM(1,q^{n-1})$, that is $RM(1,q^{n-1}) = \{xG : x \in \F_q^{n} \}$. Putting these things together we get
\begin{equation} \label{embedding RM}
 (\F_q^n, \tilde d_{BRT}) \rightarrow (RM(1,q^{n-1}), d_{\Ham}^{\,q^{n-1}}) \hookrightarrow 
 (\F_q^{q^{n-1}}, d_{\Ham}^{\,q^{n-1}}).
\end{equation}
Combining the isometry \eqref{iso hom brt} with the embedding \eqref{embedding RM} we get the isometric embedding 
 \begin{equation*} 
 (\Z_{q^n}, \tilde d_{\C}) \hookrightarrow (\F_q^{q^{n-1}}, d_{\Ham}^{\,q^{n-1}}).
 \end{equation*}
Taking $q=p$ prime, we obtain the following result of Greferath (\cite{Greferath})
 \begin{equation} \label{emb hom brt}
 (\Z_{p^n}, d_{\scriptscriptstyle Hom}) \hookrightarrow (\Z_p^{p^{n-1}}, d_{\Ham}^{\,p^{n-1}}).
 \end{equation}
In this way, similarly as in Section \ref{Isometric embeddings}, we get isometric embeddings of the form
\[
\Z_{p^n} \hookrightarrow (\F_{p^{i}}^{p^{n-i}}, d_{\Ham}^{\,p^{n-i}})
\]
for $i=1,\ldots,n$.
\end{rem}

\end{document}